\documentclass[reqno,11pt]{amsart}

\usepackage{graphicx}
\usepackage{relsize}
\usepackage[ansinew]{inputenc}
\usepackage{amsfonts,epsfig}
\usepackage{latexsym}
\usepackage{mathabx}
\usepackage{amsmath}
\usepackage{amssymb}
\usepackage{color}
\usepackage{hyperref}
\usepackage{mathrsfs}

\newtheorem{theorem}{Theorem}
\newtheorem{lemma}[theorem]{Lemma}
\newtheorem{corollary}[theorem]{Corollary}

\newtheorem{proposition}[theorem]{Proposition}

\newtheorem{lettertheorem}{Theorem}
\newtheorem{letterlemma}[lettertheorem]{Lemma}

\theoremstyle{definition}

\theoremstyle{remark}

\numberwithin{equation}{section}

\numberwithin{equation}{section}

\newcommand{\D}{\mathbb{D}}

\renewcommand{\phi}{\varphi}

\begin{document}
\title[Toeplitz and Hankel operators on Bergman spaces with doubling weights]{Toeplitz and Hankel operators on Bergman spaces with doubling weights}

	
\thanks{This work was supported by National Key R\&D Program of China (2024YFA1013400) and the National Natural Science Foundation of China (Grant Numbers 12571136, 12231005 and 12401150)}

\makeatletter
\@namedef{subjclassname@2020}{\textup{2020} Mathematics Subject Classification}
\makeatother

\author
{Yongjiang Duan}
\address
{Yongjiang Duan: Department of Mathematics,
		Jinan University, Guangzhou, Guangdong, 510632, P.R.China}
\email{yjduan@jnu.edu.cn}

\author
{Kunyu Guo}
\address
{Kunyu Guo: School of Mathematical Sciences, Fudan University, Shanghai 200433, P.R.China}
\email{kyguo@fudan.edu.cn}

\author
{Junhan Hong}
\address
{Junhan Hong: School of Mathematics and Statistics,
		Northeast Normal University, Changchun, Jilin, 130024, P.R.China}
\email{hongjh815@nenu.edu.cn}

\author
{Fugang Yan}
\address
{Fugang Yan: College of Mathematics and Statistics, Chongqing University, Chongqing 401331, P.R.China;
and Key Laboratory of Nonlinear Analysis and its Applications (Chongqing University), Ministry of Education, Chongqing 401331, P.R.China}
\email{fugangyan@cqu.edu.cn}

\begin{abstract}
In this paper, we focus on the boundedness and compactness of the Toeplitz operators and the Hankel operators on the $\mathcal{D}$-weighted Bergman spaces $A_{\omega}^p$ ($1\leq p<\infty$). In particular, the boundedness of Toeplitz operators with $\mathrm{BMO}_\omega^p$-symbols on $A_{\omega}^p$ is characterized in terms of the Berezin-type transform. Several sufficient conditions for the Toeplitz operators $T_f^\omega$ with locally integrable symbols to be bounded (resp. compact) for each $1\leq p<\infty$ are presented. Moreover, sufficient and necessary conditions for the boundedness of the Toeplitz operator $T_{\bar{f}}^\omega:A_\omega^1\rightarrow A_\omega^1$ and the Hankel operator $H_{\bar{f}}^\omega:A_\omega^1\rightarrow L_\omega^1$ with co-analytic symbol are established.
\end{abstract}
	
\maketitle

\section{Introduction and main results}

Weights arise naturally in the study of Hardy-Littlewood maximal operators on Lebesgue spaces. They are closely related to the singular integral operators in harmonic analysis. Recently, weights also played more and more important role in the study of operator theory. Let $\mathbb{D}$ be the open unit disk and $\mathrm{d}A(z)$ be the normalized area measure on $\mathbb{D}$. A weight $\omega$ on $\mathbb{D}$ is a positive integrable function, and if it satisfies $\omega(z)=\omega(|z|)$, then it is said to be a radial weight. In the past decade, several classes of radial weights on the unit disk were introduced and systematically investigated by many authors in the study of
analytic function spaces and related operator theory. In this paper, we will focus on a particular class
of doubling radial weights that have been studied extensively by Pel\'aez and R\"atty\"a.

Given a radial weight $\omega$ on $\mathbb{D},$ we say that $\omega$ is a doubling weight, denoted by $\omega \in \widehat{\mathcal{D}}$, if $\widehat{\omega}(z)=\int_{|z|}^{1}\omega(s)\mathrm{d}s$ satisfies the following condition: there exists a $C=C(\omega)\geq1$ such that
$$\widehat{\omega}(r)\leq C\widehat{\omega}\left(\frac{1+r}{2}\right),\quad{0\leqslant r<1}.$$
If there exist $K=K(\omega)>1$ and $C=C(\omega)>1$ such that
$$\widehat{\omega}(r)\geq C\widehat{\omega}\left(1-\frac{1-r}{K}\right),\quad{0\leqslant r<1},$$ then we say that $\omega$ is a reverse doubling weight, denoted by $\omega\in \widecheck {\mathcal{D}}$. Let $\mathcal{D}=\widehat{\mathcal{D}}\cap \widecheck{\mathcal{D}}$. Both $\widehat{\mathcal{D}}$ and $\mathcal{D}$ appear naturally in many instances in the operator theory of Bergman spaces induced by radial weights \cite{PR2}. Special subclasses of $\mathcal{D}$ include the collection of regular weights, denoted by $\mathcal{R}$, consisting of those continuous radial weights satisfying $\widehat{\omega}(r) \asymp\omega(r)(1-r)$ for all $0 \leq r<1$. For the basic properties of the radial weights, we refer to \cite{Pe,PR5} and the references therein. It is known that $\mathcal{R}\subsetneqq \mathcal{D}\subsetneqq \widehat{\mathcal{D}}.$

For a radial weight $\omega$ and $0<p<\infty$, we write $L^p_{\omega}=L^p(\mathbb{D}, \omega \mathrm{d}A)$. The weighted Bergman space induced by $\omega$ is defined by $A_\omega^p=L_\omega^p\cap H(\mathbb{D})$, where $ H(\mathbb{D})$ denotes the space of all analytic functions on $\mathbb{D}$. We will say that $A_\omega^p$ is a $\mathcal D$-weighted Bergman space if $\omega\in \mathcal D$. For $p=2$, $A_\omega ^2$ is a reproducing kernel Hilbert space whose reproducing kernel function is given by
$$B_z^{\omega}(\zeta)=\sum\limits_{m=0}^{\infty}\frac{(\bar{z}\zeta)^m}{2\omega_{2m+1}},$$
where $$\omega_{x}=\int_{0}^{1} t^{x} \omega(t)\mathrm{d}t,\quad 0< x<\infty.$$
In particular, if $\omega$ is the standard weight
$(\alpha+1)(1-|z|^2)^\alpha$ with $-1<\alpha<\infty$, the reproducing kernel of the standard weighted Bergman space $A^2_{\alpha}$ is given by $K_z^\alpha(\zeta)={{(1-\bar{z}\zeta)}^{-(\alpha+2)}}$. Let $P_\omega $ denote the orthogonal projection from $L^2_{\omega}$ onto $A^2_{\omega}$. Then $P_{\omega}$ admits the following integral representation:
$$
P_\omega(f)(z)=\int_{\mathbb{D}}f(\zeta)\overline{B_z^{\omega}(\zeta)}\omega(\zeta)\mathrm{d}A(\zeta), \quad f\in{L^2_{\omega}}.
$$
For more information about the weighted Bergman space $A_{\omega}^p$ and the weighted Bergman projection $P_{\omega}$, we refer to \cite{PPR,PR5,PR,PR4,PR6,PR2,PR1,PRS}.

For $f\in L_\omega^1$, the Toeplitz operator $T^{\omega}_f$ is densely defined in $A_\omega^p$ by
\begin{equation}\label{T_f}
T_f^\omega(g)(z)=\int_{\mathbb{D}}f(\zeta)g(\zeta)\overline{B_z^\omega(\zeta)}\omega(\zeta)\mathrm{d}A(\zeta), \quad g\in H^\infty,
\end{equation}
where $H^\infty$ is the space of all bounded analytic functions in $\D$. In the classical Bergman space ($\omega(z)=(\alpha+1)(1-|z|^2)^{\alpha}$ with $\alpha>-1$), the theory of the Toeplitz operators has been widely studied. We refer to \cite{Luecking, MS, Zhu3} for the Toeplitz operator with non-negative symbols and to \cite{Zorboska} for the Toeplitz operator with ${\rm BMO}$-symbols. For the general symbols, Taskinen and Virtanen \cite{TV1, TV2} obtained a sufficient condition for the Toeplitz operator $T_f$ with locally integrable symbols to be bounded on the Bergman space via some ``average" of its symbol; see also \cite{HLTV} for the higher dimension case. Recently, a sufficient condition in terms of the generalized Carleson square was given by Zheng and Yan \cite{YZ}. In the $\mathcal R$-weighted Bergman space, the boundedness and compactness of Toeplitz operators with positive symbols between $A_\omega^p$ and $A_\omega^q$ with $\omega \in \mathcal R$ was characterized by Pel\'aez et al. \cite{PRS}. Some other earlier partial results in this area include \cite{DGWW2, DGWW, DRWW}. Although there are many interesting results in this area, it is still an open problem to find a sufficient and necessary condition for the Toeplitz operator $T_f^{\omega}$ with $f\in L_{\omega}^1$ to be bounded on $A_{\omega}^p$ \cite{TV3}.

In this paper, we characterize the boundedness of Toeplitz operators with $\mathrm{BMO}_\omega^p$-symbols on $A_{\omega}^p$ induced by $\mathcal D$ weight in terms of the Berezin-type transform. We also give a sufficient condition for the Toeplitz operators $T_{f}^{\omega}$ with locally integrable symbols to be bounded (resp. compact) on $A_{\omega}^p$ ($\omega \in \mathcal D$) for each $1\leq p<\infty$. It was shown in \cite[Theorem 3 and Theorem 5]{PR2} that the class of $\mathcal{D}$ weights is the largest class of radial weights $\omega$ such that the Bergman projection $P_\omega$ satisfies the $L^\infty$-$\text{BMO}$ estimate and the Littlewood-Paley formula holds in $A_{\omega}^p$.

Before stating the main results, we introduce some additional notation. Let $\beta(\cdot,\cdot)$ denote the Bergman metric in $\mathbb D$, namely,
$$\beta(z,w)=\frac{1}{2}\log \frac{1+|\varphi_z(w)|}{1-|\varphi_z(w)|},\quad z,w\in\mathbb D,$$
where $\varphi_z(w)=\frac{z-w}{1-\bar{z}w}$. For $r>0$ and $z\in\mathbb D$, we denote by $D(z,r)$ the Bergman disc $\{{w\in \mathbb{D}:\beta(w,z)<r}\}$. For $\omega\in \mathcal{D}$, it follows from Lemma \ref{Dhat}(ii) and Lemma \ref{Dcheck} in Section 2 that there exists an $r_0=r(\omega)>0$ such that
\begin{equation}\label{Dweight2}
\omega(D(z,r)):=\int_{D(z,r)}\omega(\zeta)\mathrm{d}A(\zeta)\asymp \widehat{\omega}(z)(1-|z|),\quad{z\in\mathbb{D}},\,\,\forall\, r\geq r_0.
\end{equation}
For $r\geq r_0$ and $1 \leq p<\infty$, write
$$
\mathrm{MO}_{\omega, r}^p(f)(z)=\left(\frac{1}{\omega(D(z, r))} \int_{D(z, r)}\left|f(\zeta)-\widehat{f}_{r, \omega}(z)\right|^p \omega(\zeta) \mathrm{d}
A(\zeta)\right)^{1/p},
$$
where
\begin{equation}\label{eq1.1}
\widehat{f}_{r, \omega}(z)=\frac{\int_{D(z, r)} f(\zeta) \omega(\zeta) \mathrm{d} A(\zeta)}{\omega(D(z, r))}, \quad z \in \mathbb{D}.
\end{equation}
The space $\mathrm{BMO}_{\omega,  r}^p$ with $r\geq r_0$ consists of $f \in L_\omega^p$ such that
$$
\|f\|_{\mathrm{BMO}_{\omega, r}^p}=\sup _{z \in \mathbb{D}} \Big(\mathrm{MO}_{\omega, r}^p(f)(z)\Big)<\infty.
$$
By \cite[Theorem 11]{PPR}, for $r\geq r_0$ the norms $\|\cdot\|_{\mathrm{BMO}_{\omega, r}^p}$ and $\|\cdot\|_{\mathrm{BMO}_{\omega, r_0}^p}$ are equivalent, and thus
\begin{equation}\label{BMO1}
\mathrm{BMO}_{\omega,r}^p=\mathrm{BMO}_{\omega, r_0}^p, \quad r \geq r_0
\end{equation}
with equivalent norms. This space (independent of $r\geq r_0$) is denoted by $\mathrm{BMO}_{\omega}^p$. It will be assumed that the norm is always calculated with respect to a fixed $r \geq r_0$. It is worth noticing that if $\omega\in\widehat{\mathcal{D}}$, for a fixed $r>0$, $\omega\big(D(z,r)\big)$ may equal to zero for some $z$ close to the boundary,
see example in \cite[Proposition 3]{PR1}. This is the reason why we consider $\mathrm{BMO}_{\omega,r}^p$ with $\mathcal{D}$ weights. Notice that for $\omega\in\mathcal{R}$, straightforward calculations show that for each $r_1,r_2\in (0,\infty)$, we have $\mathrm{BMO}_{\omega,r_1}^p=\mathrm{BMO}_{\nu,r_2}^p$ with equivalent norms, where $\nu(z)\equiv 1$.

For $z,\zeta \in\mathbb{D}$ and $\alpha>-1$, define
$
k_{\omega,z}^\alpha(\zeta)=K_z^\alpha(\zeta) / \|K_z^\alpha\|_{A_\omega^2}.
$
For $g\in L_\omega^1$, the Berezin-type transform is defined as
$$
B_\omega^\alpha(g)(z)=\langle gk_{\omega,z}^\alpha, k_{\omega,z}^\alpha\rangle_{L^2_\omega}, \quad{z\in\mathbb{D}}.
$$
The first aim in this paper is to utilize the Berezin-type transform $B_\omega^\alpha$ to determine the boundedness of Toeplitz operator $T_f^\omega$ with $\mathrm{BMO}_{\omega,r}^p$-symbols on the $\mathcal{D}$-weighted Bergman space $A_\omega^p$.

\begin{theorem}\label{TB}
Let $\omega\in \mathcal{D}$ and $1<p<\infty$. For $f\in \mathrm{BMO}^p_{\omega}$, the Toeplitz operator $T_f^\omega$ is bounded on $A_\omega^p$ if and only if there exists $\alpha_0=\alpha(\omega,p)>0$ such that $B_\omega^\alpha(f)$ is bounded in $\D$ for any
$\alpha>\alpha_0.$
\end{theorem}
The key to prove Theorem \ref{TB} is to establish a connection between $\mathrm{BMO}^p_{\omega,r}$  and the Berezin-type transform $B_\omega^\alpha$ (see Proposition \ref{BMOr}). In \cite{KW}, Keshavarzi and Wu introduced the global mean oscillations via a certain Berezin-type transform $\mathrm{MO}_{\omega,\alpha}(f)$ and use it to describe the membership in Schatten class of the Hankel operators. Inspired by this, for $1\leq p <\infty$ we define
$$
\mathcal{ MO}_{\omega,\alpha}^p(f)(z)=\left(\int_{\mathbb{D}}\left|f(\zeta)-B_\omega^\alpha(f)(z)\right|^p |k_{\omega,z}^\alpha(\zeta)|^2
\omega(\zeta)\mathrm{d}A(\zeta)\right)^{1/p}.
$$
Let $ \mathcal{BMO}_{\omega,\alpha}^{p}$ denote the space of all $f \in L_\omega^p$ such that
$$
\|f\|_{ \mathcal{BMO}_{\omega,\alpha}^{p}}=\sup _{z \in \mathbb{D}} \Big(\mathcal{ MO}_{\omega,\alpha}^p(f)(z)\Big)<\infty.
$$
It turns out that $\mathcal{BMO}_{\omega,\alpha}^{p}$ coincides with $\operatorname{ BMO}_{\omega,r}^{p}$ when $r$ and $\alpha$ are large enough (see Proposition \ref{BMOeta}). Based on this identification and some idea of Zorboska \cite{Zorboska}, Theorem \ref{TB} is proved.

The second aim of this paper is to give some sufficient conditions for the boundedness and compactness of $T_f^\omega$ with locally integrable symbols on $A_\omega^p$. For $0<h<1$, $\theta\in[0,2\pi]$ and $0<\tau\leq1$, we define the $\tau$-Carleson square $S_h^\tau(e^{i\theta})$ to be
$$
S_h^\tau(e^{i\theta})=\{re^{it}:1-h<r<1,|t-\theta|\leq \pi\tau h\}.
$$
It is easy to verify that $\omega(S_{1-h}^{1}(e^{i\theta}))\asymp \widehat{\omega}(h)(1-h)$. For $f\in L_{\rm{{loc}}}^1(\mathbb{D},\omega \mathrm{d}A)$, $0<r<1$, $\theta\in[0,2\pi]$ and $0<\tau\leq1$, we define a weighted mean of $f$ to be
$${\bf M}(f,r,\theta,\tau)=\frac{1}{\widehat{\omega}(r)(1-r)}\int_{S^\tau_{1-r}(e^{i\theta})}f(z)\omega(z)\mathrm{d}A(z).$$
The following theorem is concerned with the case $1<p<\infty$.

\begin{theorem} \label{main1}
Let $\omega\in\mathcal{D}$ and $f\in L_{\rm{{loc}}}^1(\mathbb{D},\omega \mathrm{d}A)$. Then the following conclusions hold.
\begin{enumerate}
\item [(i)] If
\begin{equation}\label{A1}
\sup\Big\{\big|{\bf M}(f,r,\theta,\tau)\big|:0<r<1,\,\,\theta\in[0,2\pi],\,\, 0<\tau\leq1\Big\}<\infty,
\end{equation}
then $T_f^\omega: A_\omega^p \rightarrow A_\omega^p$ is well-defined by \eqref{T_f} and bounded for $1<p<\infty$.

\item [(ii)]  If
\begin{equation}\label{B0}
\lim_{r\rightarrow1^-}\sup\Big\{\big|{\bf M}(f,r,\theta,\tau)\big|:\theta\in[0,2\pi],\,\, 0<\tau\leq1\Big\}=0,
\end{equation}
then $T_f^\omega: A_\omega^p \rightarrow A_\omega^p$ is compact for $1<p<\infty$.
\end{enumerate}
\end{theorem}

\noindent If the symbol is non-negative, then the conditions in Theorem \ref{main1} are also necessary.

To deal with the case $p=1$, for $f\in L_{\rm{{loc}}}^1(\mathbb{D},\omega \mathrm{d}A)$ we introduce the log-weighted mean of $f$ defined by
$$\mathscr{M}(f,r,\theta,\tau)=\frac{
\log{\frac{e}{1-r}}}{\widehat{\omega}(r)(1-r)}\int_{S^\tau_{1-r}(e^{i\theta})}f(z)\omega(z)\mathrm{d}A(z).$$

\begin{theorem} \label{main2}
Let $\omega\in\mathcal{D}$ and $f\in L_{\rm{{loc}}}^1(\mathbb{D},\omega \mathrm{d}A)$. Then the following conclusions hold.
\begin{enumerate}
\item [(i)] If
\begin{equation}\label{E0}
\sup\Big\{\big|\mathscr{M}(f,r,\theta,\tau)\big|:0<r<1,\,\,\theta\in[0,2\pi],\,\, 0<\tau\leq1\Big\}<\infty,
\end{equation}
then $T_f^\omega$ is well-defined by \eqref{T_f} and bounded on $A_\omega^1$.

\item [(ii)] If
$$\lim_{r\rightarrow1^-}\sup\Big\{\big|\mathscr{M}(f,r,\theta,\tau)\big|:\theta\in[0,2\pi],\,\,\tau\in(0,1]\Big\}=0,$$
then $T_f^\omega: A_\omega^1 \rightarrow A_\omega^1$ is compact.
\end{enumerate}
\end{theorem}

We require the weight $\omega$ in Theorem \ref{main1} and \ref{main2} to be in $\mathcal{D}$ because the Littlewood-Paley formula, which is frequently used in the proof, holds in $A_{\omega}^p$ if and only if $\omega\in \mathcal{D}$ \cite[Theorem 5]{PR2}. In addition, we will use the estimate of the $A_\omega^p$-norms of the $k$-th derivative of the reproducing kernel $B_z^\omega$ with $\omega \in\widehat{\mathcal{D}}$ in our arguments.
\vspace{.2cm}

Another problem considered in this paper is to characterize the boundedness of the Toeplitz operators with co-analytic symbols on $A_\omega^1$. To do this, let $\mathcal{LB}$ denote the logarithm-Bloch space consisting of $f\in H(\mathbb{D})$ which satisfies (see \cite{Attele})
$$
\sup_{z\in\mathbb{D}}|f^{\prime}(z)|(1-|z|^2)\log \frac{e}{1-|z|}<\infty.
$$
For $\beta>0$, let $$\widetilde{B}_{\zeta}^{\omega_{[\beta]}}(z)=B_{\zeta}^{\omega_{[\beta]}}(z)(1-|\zeta|)^\beta,\quad{\zeta,z\in\mathbb{D}},$$
where $\omega_{[\beta]}(\zeta)=\omega(\zeta)(1-|\zeta|)^\beta$. The necessary and sufficient conditions for the Toeplitz operator to be bounded on $A_\omega^1$ are provided. For the classical Bergman space, we refer to \cite{ABT, Zhu1}.

\begin{theorem} \label{TLB}
Let $\omega\in\mathcal{D}$ and $f\in A_\omega^1$. Then the following assertions are equivalent:
\begin{enumerate}
\item [(i)] $T_{\bar{f}}^\omega$ is bounded on $A_\omega^1$;
\item [(ii)] $f$ belongs to $H^\infty\cap\mathcal{LB}$;
\item [(iii)] $\sup_{\zeta \in\mathbb{D}}\| T_{\bar{f}}^\omega\widetilde{B}_{\zeta}^{\omega_{[\beta]}}\|_{A_\omega^1}<\infty.$
\end{enumerate}
\end{theorem}
\noindent The main tool to prove Theorem \ref{TLB} is the duality relation $(A_\omega^1)^*\simeq \mathcal{B}$ via the $A_\omega^2$-pairing for $\omega\in\mathcal{D}$ developed in \cite{PR2}, where $\mathcal{B}$ is the Bloch space. And the method used here does not work if the weight $\omega$ is only assumed to belong to $\widehat{\mathcal{D}}$.
\vspace{.2cm}

Besides the study of the Toeplitz operator on $A_{\omega}^1$, we also investigate the boundedness of the Hankel operator $H_f^{\omega}$ from $A_\omega^1$ to $L_\omega^1$ which is densely defined by
$$
H_f^\omega(g)(z)=(I-P_\omega)(fg)(z),\quad g\in H^{\infty},\quad z\in\mathbb{D},
$$
where the symbol $f$ belongs to $L_{\omega}^2$. Earlier results related to the Hankel operators in $A_{\omega}^p$ with $p>1$ include \cite{Axler, BBCZ, BCZ, HL, Luecking1, LZ, PZZ, PPR, Zhu, Zhu2}. In this context, we focus on the boundedness of the Hankel operators from $A_\omega^1$ to $L_\omega^1$. For the classical case, Attele \cite{Attele} provided a sufficient and necessary condition for the boundedness of the Hankel operator from $A^1$ to $L^1$.
In \cite{TV2}, Taskinen and Virtanen showed that Hankel operator with the symbol in $\mathrm{BO}_{\log} \cap L^{\infty}+\mathrm{BA}^1_{\log}$ (but not $\mathrm{BMO}^1_{\log}$) is bounded from $A^1$ to $L^1$.
The boundedness of the Hankel operator with co-analytic symbols from $A_\omega^1$ to $L_\omega^1$ is characterized by the following theorem.

\begin{theorem}\label{HLB}
Let $\omega \in \mathcal{D}$. For $f\in A_\omega^2$, the Hankel operator $H_{\bar{f}}^\omega$ is bounded from $A_\omega^1$ to $L_\omega^1$ if and only if
$f\in \mathcal{LB}$.
\end{theorem}
Interestingly, the sufficient and necessary condition in Theorem \ref{HLB} is independent of the weight $\omega$. This is in sharp contrast to the characterization for the boundedness of $H_f$ and $ H_{\bar{f}}$ from $A_\omega^p$ to $L_\omega^q$ for $1< p,q<\infty$ in \cite{HL, PPR}, where the condition relies on the weight. The proof of the necessity in Theorem \ref{HLB} is different from the case of the classical Bergman space. In \cite{Attele}, a useful identity with the Begman projection $P$ is
$$ P\Big((1-|w|^2)(\bar{w})^{-1}g^{\prime}(w)\Big)(z)=g(z),\quad{z\in\mathbb{D}},$$
for all $g\in A^2$ satisfying $g(0)=g^{\prime}(0)=0$. For the radial weight $\omega\in\mathcal{D}$, in addition to the standard arguments, an extra technique is to estimate the $L^2_\omega$-norms of $|(1-\bar{\zeta}z)^{\varepsilon}(B_{\zeta}^\omega)^{\prime}(z)|$ for $z\in\mathbb{D}$ where $\varepsilon=\varepsilon(\omega) $ is chosen sufficiently small. The proof of the sufficiency in Theorem \ref{HLB} is technically demanding and elaborate. Instead of dealing with the Hankel operator $H_f^\omega$ directly, we make use of another integral operator $S_f^\omega$ defined by
$$
S_f^\omega(h)(z)=\int_{\mathbb{D}}\Big(f(\zeta)-f(z)\Big)\overline{B_z^\omega(\zeta)}h(\zeta)\omega(\zeta)\mathrm{d}A(\zeta),\quad z\in \mathbb{D},\,\,h\in L^\infty.
$$
Moreover, we will show that $S_f^\omega$ is bounded from $L^\infty$ to $\mathcal{B}$ and
$$\|H_{\bar{f}}\|_{A_\omega^1\rightarrow L_\omega^1}\lesssim \|S_f^\omega\|_{L^\infty\rightarrow \mathcal{B}}.$$

\noindent {\bf Notation.} We denote by $a\lesssim b$ if there exists a positive constant $C=C(\cdot)$ such that $a\leq Cb$, here the constant $C(\cdot)$ depends on the parameters indicated in the parenthesis, varying under different circumstances. And we also denote by $a\asymp b$ if both $b\lesssim a$ and $b\gtrsim a$
hold.

\section{Some preliminary results about the weighted Bergman spaces}
In this section, we collect some necessary results about the weights and the weighted Bergman spaces that will be frequently used in the sequel. Throughout this paper, we will assume $\widehat{\omega}(r)=\int_{r}^1\omega(s)\mathrm{d}s>0$ for all $0\leq r<1$.

The first result provides several basic properties of weights in the class $\widehat{\mathcal{D}},$ see \cite[Lemma 2.1]{Pe}. Recall that for a radial weight $\omega$ we write $\omega_{[\beta]}(\zeta)=\omega(\zeta)(1-|\zeta|)^\beta$ ($\beta\in\mathbb R$) and
 $$\omega_x=\int_{0}^{1} t^{x} \omega(t)\mathrm{d}t, \quad 0< x<\infty.$$

\begin{letterlemma} \label{Dhat}
Let $\omega$ be a radial weight. Then the following statements are equivalent:
	\begin{itemize}
		\item[(i)] $\omega\in \widehat{\mathcal{D}}$;
		\item[(ii)] There exist constants $C=C(\omega)\geq1$ and $\beta=\beta(\omega)>0$ such that
		$$\widehat{\omega}(r)\leq C\left(\frac{1-r}{1-t}\right)^\beta\widehat{\omega}(t),
		\quad{0\leq r\leq t<1};$$
		\item[(iii)] There exist $\lambda=\lambda(\omega)\geq 0$ such that
		$$
		\int_{\mathbb{D}}\frac{\omega(z)\mathrm{d}A(z)}{|1-\bar{\zeta}z|^{\lambda+1}}\asymp
\frac{\widehat{\omega}(\zeta)}{(1-|\zeta|)^\lambda},\quad{\zeta\in\mathbb{D}};
		$$
		\item[(iv)] For each $\beta>0$, there exists a constant $C=C(\omega,\beta)>0$ such that
		$$
		x^{\beta}(\omega_{[\beta]})_x \leq C \omega_x,\quad 0\leq x<\infty;
		$$
        \item[(v)] There exist $C=C(\omega)>0$ and $\eta=\eta(\omega)>0$ such that
        $$
        \omega_x\leq C\bigg(\frac{y}{x}\bigg)^\eta\omega_y,\quad{0<x\leq y<\infty}.
        $$
	\end{itemize}
\end{letterlemma}

Similarly, the weights in $\widecheck{\mathcal{D}}$ enjoy the following properties, see \cite[Lemma B]{PPR}.

\begin{letterlemma}\label{Dcheck}
Let $\omega$ be a radial weight. Then $\omega \in \widecheck{\mathcal{D}}$ if and only if there exist $C=C(\omega)>0$ and $\alpha=\alpha(\omega)>0$ such
that
	$$
	\widehat{\omega}(t) \leqslant C\left(\frac{1-t}{1-r}\right)^\alpha \widehat{\omega}(r), \quad 0 \leqslant r \leqslant t<1.
	$$
\end{letterlemma}
The following estimation about the weights is also needed, which is from \cite{PR1}.
\begin{letterlemma}\label{Dweight1}
	Let $\omega \in \mathcal{D}$ and $\nu \in \widehat{\mathcal{D}}$. Then there exists $\gamma_0=\gamma_0(\omega, \nu)>0$ such that for each $\gamma
\in\left(0, \gamma_0\right]$, we have $(\widehat{\nu})^{-\gamma} \omega \in \mathcal{D}$, and
	$$
	\int_r^1 \frac{\omega(s)}{\widehat{\nu}(s)^\gamma} \mathrm{d} s \asymp \frac{\widehat{\omega}(r)}{\widehat{\nu}(r)^\gamma}, \quad 0 \leq r<1.
	$$
\end{letterlemma}

For $0<r<1$ and $f\in H(\mathbb{D})$, set
$$
M_p(r,f)=\left(\frac{1}{2\pi}\int_{0}^{2\pi}|f(re^{i\theta})|^p \mathrm{d}\theta\right)^{1/p}.
$$
The next lemma gives some useful estimates of the reproducing kernels of the weighted Bergman spaces induced by $\widehat{\mathcal{D}}$-weights, which can be referred to \cite[Theorem 1]{PR4}.
\begin{letterlemma}\label{Ne1}
	Let $\omega \in \widehat{\mathcal{D}}$, $0<p<\infty$, and
	$n \in \mathbb{N} \cup\{0\}$. Then the following assertions hold:
	\begin{itemize}
		\item [(i)]$M_p^p\left(r,(B_z^\omega)^{(n)}\right)\asymp\int_{0}^{|z|r} \frac{\mathrm{d}t}{\widehat{\omega}(t)^p(1-t)^{p(n+1)}}, \quad
r,|z|\rightarrow 1^-.$
		\item [(ii)] If $\nu\in \widehat{\mathcal{D}}$, then
		$$\left\|(B_{z}^{\omega}\right)^{(n)}\|_{A_{\nu}^{p}}^{p}\asymp\int_{0}^{|z|}
\frac{\widehat{\nu}(t)}{\widehat{\omega}(t)^{p}(1-t)^{p(n+1)}}\mathrm{d}t,
		\quad{|z| \rightarrow 1^{-}}.$$
		In particular, if $1<p<\infty$, then
		$$\left\|B_{z}^{\omega}\right \|_{A_{\omega}^{p}}^{p}\asymp
		\frac{1}{\Big(\widehat{\omega}(z)(1-|z|)\Big)^{p-1}}, \quad{z \in \mathbb{D}}.$$
	\end{itemize}
\end{letterlemma}
For the weighted Bergman spaces, it turns out that the class of $\mathcal{D}$ weights is the largest class of radial weights $\omega$ such that the Littlewood-Paley formula holds for $A_{\omega}^p$ (see \cite{PR2} for details).
\begin{letterlemma}\label{lemE}
Let $\omega$ be a radial weight, $0<p<\infty$ and $k\in \mathbb{N}$. Then
\begin{equation}\label{Littlewood}
\|f\|_{A_\omega^p}\asymp \int_{\mathbb{D}} |f^{(k)}(z)|^p(1-|z|)^{kp}\omega(z)\mathrm{d}A(z)+ \sum_{j=0}^{k-1} |f^{(j)}(0)|^p,\quad{f\in H(\mathbb{D})},
\end{equation}
if and only if $\omega\in \mathcal{D}$.
\end{letterlemma}
Let
$$
\widetilde\omega(r)=\frac{\widehat\omega(r)}{1-r},\quad{0 \leq r \leq 1.}
$$
Then the following result gives the useful equivalent relationship of analytic functions in the two weighted spaces (see \cite[Proposition 5]{PRS}).

\begin{letterlemma}\label{DR}
Let $\omega \in \mathcal{D}$ and $0<p<\infty$. Then $\widetilde{\omega}\in \mathcal{R}$, and $\widehat{\widetilde\omega}\asymp\widehat\omega$. Moreover,
\begin{equation}\label{RD}
\| f \|_{A_{\widetilde\omega}^{p}}\asymp\|f \|_{A_{\omega}^{p}}
\end{equation}
holds for all $f \in H(\mathbb{D})$.
\end{letterlemma}

\section{The characterizations of bounded Toeplitz operators}
\subsection{The characterization of $\mathrm{BMO}_{\omega,r}^p$ and the Berezin-type transform.}
At the beginning, we recall the following two spaces that will be frequently used in the sequel. For $r>0$, let $\mathrm{BO}$ denote the space of all continuous $f:\mathbb{D}\rightarrow \mathbb{C}$ such that
$\sup_{z\in\mathbb{D}} \Omega_r f(z)<\infty,$
where
$$\Omega_r f(z)=\sup \Big\{|f(z)-f(w)|:\beta(z,w)<r\Big\}.$$
It is well-known that (see \cite[Chapter 8]{Zhu2}) the definition of $\mathrm{BO}$ is independent of the choice of $r$, and $f\in \mathrm{BO}$ if and only if
\begin{equation}\label{BO}
|f(z)-f(w)|\lesssim \|f\|_{\mathrm{BO}}\Big(1+\beta(z,w)\Big),\quad{z,w\in\mathbb{D}}.
\end{equation}
For $0<p<\infty$, we denote by $\mathrm{BA}_{\omega,r}^p$ the space of functions $f\in L_\omega^p$ with
$$
\|f\|_{\mathrm{BA}_{\omega,r}^p}=\sup\left\{\Big(\widehat{(|f|^p)}_{r,\omega}(z)\Big)^\frac{1}{p}:z \in \mathbb{D}\right\}<\infty,
$$
where $\widehat{(|f|^p)}_{r,\omega}(z)$ is defined by \eqref{eq1.1}. Actually, if $\omega\in\mathcal{D}$, it follows from \cite[Lemma 10]{PPR} that  $\mathrm{BA}_{\omega,r}^p=\mathrm{BA}_{\omega,r_0}^p$ with equivalent norms for all $r\geq r_0$, where $r_0$ is chosen in \eqref{Dweight2}. We denote this space by $\mathrm{BA}_{\omega}^p$ which is independent of $r\geq r_0$. By Lemma \ref{Dhat}(iii), there exists $\alpha_1=\alpha(\omega)$ such that
\begin{equation}\label{Dweight3}
|k_{\omega,z}^\alpha(\zeta)|\asymp \frac{(1-|z|)^{\alpha+3/2}}{\widehat{\omega}(z)^{1/2}|1-\bar{z}{\zeta}|^{\alpha+2}},\quad{z,\zeta \in\mathbb{D}}
\end{equation}
for $\alpha>\alpha_1$. Recall that a sequence $\{a_j\}$ in $\mathbb{D}$ is called an $r$-lattice in the Bergman metric if $\mathbb{D}=\bigcup_{j=1}^\infty D(a_j,r)$ and $\beta(a_i,a_j)\geq r/2$ for $i\neq j$, see \cite{Zhu2} for details. For $\omega\in \mathcal{D}$, the space $\mathrm{BA}_{\omega}^p$ can be characterized by the following lemma.

\begin{lemma}\label{BA}
Let $\omega\in \mathcal{D}$ and $1\leq p<\infty$. Then $f$ belongs to $\mathrm{BA}_{\omega}^p$ if and only if there exists an $\alpha_1=\alpha(\omega)>0$ such that $\sup_{z\in\mathbb{D}}B_\omega^\alpha(|f|^p)(z)<\infty$ for all $\alpha>\alpha_1$.
\end{lemma}
\begin{proof}
Choose $r_0$ as in \eqref{Dweight2} and $\lambda(\omega)$ as in Lemma \ref{Dhat}(iii) with $\lambda(\omega)>1$. Firstly assume $f\in \mathrm{BA}^p_{\omega}=\mathrm{BA}_{\omega,r}^p$ for
$r\geq {r_0}$.
Take an $r$-lattice $\{a_j\}$ as in \cite[Lemma 4.7]{Zhu2}. For each $\alpha>\alpha_1(\omega)=(\lambda(\omega)-3)/2$, Lemma \ref{Dhat}(iii), \eqref{Dweight2}, \eqref{RD} and \eqref{Dweight3} together give
$$
\begin{aligned}
B_\omega^\alpha(|f|^p)(z)
& \lesssim\frac{(1-|z|)^{2\alpha+3}}{\widehat{\omega}(z)}\sum_{j=1}^\infty \int_{D(a_j,r)}
\frac{|f(\zeta)|^p}{|1-\bar{z}\zeta|^{2\alpha+4}}\omega(\zeta)\mathrm{d}A(\zeta) \\
& \lesssim \frac{(1-|z|)^{2\alpha+3}}{{\widehat{\omega}(z)}}\|f\|_{\mathrm{BA}_{\omega}^p}^p\sum_{j=1}^\infty\frac{\omega(D(a_j,r))}{|1-\bar{z}a_j|^{2\alpha+4}} \\
& \lesssim \|f\|_{\mathrm{BA}_{\omega}^p}^p\frac{(1-|z|)^{2\alpha+3}}{{\widehat{\omega}(z)}}\int_{\mathbb{D}}\frac{\widetilde{\omega}(u)}{|1-\bar{z}u|^{2\alpha+4}}
\mathrm{d}A(u) \\
& \asymp \|f\|_{\mathrm{BA}_{\omega}^p}^p\frac{(1-|z|)^{2\alpha+3}}{{\widehat{\omega}(z)}}\int_{\mathbb{D}}\frac{\omega(u)}{|1-\bar{z}u|^{2\alpha+4}}\mathrm{d}A(u)\\
&\asymp
\|f\|_{\mathrm{BA}_{\omega}^p}^p,\quad{z\in\mathbb{D}},
\end{aligned}
$$
which shows that $\sup_{z\in\mathbb{D}}B_\omega^\alpha(|f|^p)(z)<\infty$.

Conversely, assume $\sup_{z\in\mathbb{D}}B_\omega^\alpha(|f|^p)(z)<\infty$ for $\alpha>\alpha_1$. Let $r_0$ be as in \eqref{Dweight2} and $r\geq r_0$. Then \eqref{Dweight2} together with \eqref{Dweight3} yields
$$
\begin{aligned}
B_\omega^\alpha(|f|^p)(z)
& \gtrsim \frac{(1-|z|)^{2\alpha+3}}{{\widehat{\omega}(z)}}\int_{D(z,r)}\frac{|f(\zeta)|^p}{|1-\bar{z}\zeta|^{2\alpha+4}}\omega(\zeta)\mathrm{d}A(\zeta) \\
& \asymp \frac{1}{\widehat{\omega}(z)(1-|z|)} \int_{D(z,r)}|f(\zeta)|^p \omega(\zeta)\mathrm{d}A(\zeta) \\
& \asymp \frac{1}{\omega(D(z,r))} \int_{D(z,r)}|f(\zeta)|^p \omega(\zeta)\mathrm{d}A(\zeta)\\
&=\widehat{(|f|^p)}_{r,\omega}(z),\quad z\in \mathbb{D},
\end{aligned}
$$
which completes the proof.
\end{proof}
Next we investigate the relationship between $\mathrm{BMO}_{\omega}^p$ and the Berezin-type transform $B_\omega^\alpha$.
\begin{proposition}\label{BMOr}
Let $\omega\in \mathcal{D}$ and $1\leq p<\infty$. Then  $f$ belongs to $\mathrm{BMO}_{\omega}^p$ if and only if
there exists an $\alpha_0=\alpha(\omega,p)>0$ such that
$$\sup_{z\in\mathbb{D}}B_\omega^\alpha\Big(|f-\widehat{f}_{r,\omega}(z)|^p\Big)(z) <\infty$$
for all $\alpha>\alpha_0$.
\end{proposition}
\begin{proof}
We only give the proof of the case of $p>1$, the case of $p=1$ is similar.
Assume $f \in \mathrm{BMO}_{\omega}^p$. Then $f$ belongs to $\mathrm{BMO}_{\omega,r}^p$ for each $r\geq r_0$ where $r_0$ is given in \eqref{Dweight2}. Then  $f$ can be written as $f=f_1+f_2$, where $f_1\in
\mathrm{BA}_{\omega,r}^p$ and $f_2\in \mathrm{BO}$ (see \cite[Theorem 11]{PPR} for details).
Choose $\lambda(\omega)$ as in Lemma \ref{Dhat}(iii) with $\lambda(\omega)\geq1$.
Let $\gamma_0(\omega)$ be as in Lemma \ref{Dweight1}. Choose an $\varepsilon=\varepsilon(\omega)>0$ such that $p\varepsilon <\gamma_0(\omega)$. Then we have
$\omega_{[-p\varepsilon]}\in \mathcal{D}$ and $\widehat{\omega_{[-p\varepsilon]}}(r)\asymp\widehat{\omega}(r)/(1-r)^{p\varepsilon}$ for $0\leq r<1$. Let $\alpha_0(\omega,p)=p\varepsilon +(\lambda(\omega)-3)/2$.
Since $f_1\in \mathrm{BA}_{\omega,r}^p$, an application of the triangle inequality and  H\"{o}lder's inequality gives
\begin{eqnarray}\label{ba}
B_{\omega}^\alpha\Big(|f_1-\widehat{(f_1)}_{r,\omega}(z)|^p\Big)(z)
& \lesssim& \int_{\mathbb{D}}|f_1(\zeta)|^p|k_{\omega,z}^\alpha(\zeta)|^2\omega(\zeta)\mathrm{d}A(\zeta)+|\widehat{(f_1)}_{r,\omega}(z)|^p \nonumber\\
& \leq &B_{\omega}^\alpha(|f_1|^p)(z)+\widehat{(|f_1|^p)}_{r,\omega}(z),  \quad{z\in\mathbb{D}}.
\end{eqnarray}
Then by Lemma \ref{BA} and the definition of $\mathrm{BA}_{\omega,r}^p$, $B_{\omega}^\alpha\Big(|f_1-\widehat{(f_1)}_{r,\omega}(z)|^p\Big)(z)$ is bounded on
$\mathbb{D}$.
Next, we estimate $B_{\omega}^\alpha\Big(|f_2-\widehat{(f_2)}_{r,\omega}(z)|^p\Big)(z)$.
By the definition of $\widehat{(f_2)}_{r,\omega}$, H\"{o}lder's inequality and \eqref{BO}, for any $\alpha>\alpha_0(\omega,p)$ we have
\begin{eqnarray}\label{bo}
I(z)
&:=& \int_{\mathbb{D}}|f_2(\zeta)-\widehat{(f_2)}_{r,\omega}(z)|^p|k_{\omega,z}^\alpha(\zeta)|^2\omega(\zeta)\mathrm{d}A(\zeta)\nonumber \\
& \leq& \int_{\mathbb{D}}\left(\frac{1}{\omega(D(z,r))}\int_{D(z,r)}|f_2(\zeta)-{f_2}(u)|^p\omega(u)
\mathrm{d}A(u)\right)|k_{\omega,z}^\alpha(\zeta)|^2\omega(\zeta)\mathrm{d}A(\zeta) \nonumber\\
& \lesssim& \|f_2\|_{\mathrm{BO}}^p\left( 1+\int_{\mathbb{D}}\frac{1}{\omega(D(z,r))}\int_{D(z,r)}\beta(u,\zeta)^p\omega(u)
\mathrm{d}A(u)|k_{\omega,z}^\alpha(\zeta)|^2\omega(\zeta)\mathrm{d}A(\zeta)\right).
\end{eqnarray}
Notice that
\begin{equation}\label{beta}
\beta(z,\zeta)\lesssim \left(\frac{|1-\bar{z}\zeta|^2}{(1-|z|)(1-|\zeta|)}\right)^\varepsilon,\quad{z,\zeta\in \mathbb{D}}.
\end{equation}
From Lemma \ref{Dhat}(iii), \eqref{Dweight3},
\eqref{bo}, \eqref{beta} and \cite[Lemma 4.30]{Zhu2}, we deduce
\begin{eqnarray}
I(z)& \lesssim&  \|f_2\|_{\mathrm{BO}}^p\left(1+ \frac{(1-|z|)^{2\alpha+3}}{\widehat{\omega}(z)\omega(D(z,r))}\int_{\mathbb{D}}\int_{D(z,r)}\frac{|1-\bar{u}\zeta|^{2p\varepsilon}}
{(1-|u|)^{p\varepsilon}}\frac{\omega(u)}{|1-\bar{z}\zeta|^{2\alpha+4}} \mathrm{d}A(u)\frac{\omega(\zeta)}{(1-|\zeta|)^{p\varepsilon}}\mathrm{d}A(\zeta)\right)\nonumber \\
& \asymp & \|f_2\|_{\mathrm{BO}}^p\left(1+
\frac{(1-|z|)^{2\alpha+3-p\varepsilon}}{{\widehat{\omega}(z)}}\int_{\mathbb{D}}\frac{\omega_{[-p\varepsilon]}(\zeta)}{|1-\bar{z}\zeta|^{2\alpha+4-2p\varepsilon}}
\mathrm{d}A(\zeta)\right)\nonumber\\
&\lesssim&\|f_2\|_{\mathrm{BO}}^p\left(1+
\frac{\widehat{\omega_{[-p\varepsilon]}}(z)(1-|z|)^{p\varepsilon}}{{\widehat{\omega}(z)}}\right)\nonumber\\
&\lesssim&
\|f_2\|_{\mathrm{BO}}^p, \quad{z\in \mathbb{D}}.\label{eq3.6}
\end{eqnarray}
Combining \eqref{ba}, \eqref{eq3.6} and the triangle inequality gives $\sup_{z\in\mathbb{D}}B_\omega^\alpha\Big(|f-\widehat{f}_{r,\omega}(z)|^p\Big)(z) <\infty.$

Conversely, suppose there exists $\alpha_0$ such that $\sup_{z\in\mathbb{D}}B_{\omega}^\alpha\Big(|f-\widehat{f}_{r,\omega}(z)|^p\Big)(z)<\infty$ for all $\alpha>\alpha_0$. Let $r_0$ be as in \eqref{Dweight2}. Then for $r\geq r_0$, \eqref{Dweight2} and \eqref{Dweight3} yield
$$
\begin{aligned}
\Big(\mathrm{MO}_{\omega,r}^p(f)(z)\Big)^p
&\asymp
\frac{(1-|z|)^{2\alpha+3}}{{\widehat{\omega}(z)}}\int_{D(z,r)}|f(\zeta)-\widehat{f}_{r,\omega}(z)|^p\frac{\omega(\zeta)}{|1-\bar{z}\zeta|^{2\alpha+4}}\mathrm{d}A(\zeta)
\\
& \lesssim B_{\omega}^\alpha\Big(|f-\widehat{f}_{r,\omega}(z)|^p\Big)(z),\quad z\in\mathbb{D},
\end{aligned}
$$
which shows that $f\in \mathrm{BMO}_{\omega,r}^p$.
The proof is completed.
\end{proof}
\begin{proposition}\label{BMOeta}
Let $\omega\in\mathcal{D}$ and $1\leq p<\infty$. There exists an $\alpha_0=\alpha(\omega,p)$ such that  $$\mathrm{BMO}_{\omega}^p =
\mathcal{ BMO}_{\omega,\alpha}^p, \quad{\forall\, \alpha>\alpha_0.}$$
\end{proposition}
\begin{proof}
Let $r_0$ be as in \eqref{Dweight2} and $\alpha_0$ as in Proposition \ref{BMOr}.
Let $1<p<\infty$. Suppose $\alpha>\alpha_0$ is such that $f\in \mathcal{ BMO}_{\omega,\alpha}^p$. Then, \eqref{Dweight2} and \eqref{Dweight3} imply
\begin{equation}\label{BMO2}
\begin{aligned}
\Big({\mathcal{ MO}_{\omega,\alpha}^p(f)(z)}\Big)^p
&
\asymp\int_{\mathbb{D}}\left|f(\zeta)-B_\omega^\alpha(f)(z)\right|^p\frac{(1-|z|)^{2\alpha+3}}{\widehat{\omega}(z)|1-\bar{z}{\zeta}|^{2\alpha+4}}\omega(\zeta)\mathrm{d}A(\zeta)
\\
& \gtrsim \frac{1}{\widehat{\omega}(z)(1-|z|)}\int_{D(z,r)}\left|f(\zeta)-B_\omega^\alpha(f)(z)\right|^p\omega(\zeta)\mathrm{d}A(\zeta) \\
& \asymp \frac{1}{\omega(D(z,r))}\int_{D(z,r)}\left|f(\zeta)-B_\omega^\alpha(f)(z)\right|^p\omega(\zeta)\mathrm{d}A(\zeta)
\end{aligned}
\end{equation}
for $r\geq r_0$ and $z\in\mathbb{D}$. By the triangle inequality and H\"{o}lder's inequality, we have
\begin{align}\label{BMO3}
\Big(\mathrm{ MO}_{\omega,r}^p(f)(z)\Big)^p
&\lesssim\frac{1}{\omega(D(z,r))}\int_{D(z,r)}\left|f(\zeta)-B_\omega^\alpha(f)(z)\right|^p\omega(\zeta)\mathrm{d}A(\zeta)+|B_\omega^\alpha(f)(z)-\widehat{f}_{r, \omega}(z)|^p\nonumber\\
&\lesssim \frac{1}{\omega(D(z,r))}\int_{D(z,r)}\left|f(\zeta)-B_\omega^\alpha(f)(z)\right|^p\omega(\zeta)\mathrm{d}A(\zeta),\quad{z\in\mathbb{D}}.
\end{align}
And for the case $p=1$, the last line is obvious.
Combining \eqref{BMO2} with \eqref{BMO3}, for all $r\geq r_0$ we have $\mathcal{ BMO}^p_{\omega,\alpha}\subseteq \mathrm{BMO}^p_{\omega,r}$.
Now we let $f\in  \mathrm{BMO}_{\omega}^p=\mathrm{BMO}_{\omega,r}^p$ with $r\geq r_0$. Then for $1<p<\infty$, the triangle inequality and H\"{o}lder's inequality
yield
$$
\begin{aligned}
\Big(\mathcal{ MO}_{\omega, \alpha}^p(f)(z)\Big)^p
 & \lesssim \int_\mathbb{D}|f(\zeta)-\widehat{f}_{r,\omega}(z)|^p|k_{\omega,z}^\alpha(\zeta)|^2\omega(\zeta)\mathrm{d}A(\zeta)
 +|\widehat{f}_{r,\omega}(z)-B_\omega^\alpha{f}(z)|^p \\
 &\lesssim B_\omega^\alpha\Big(|f-\widehat{f}_{r,\omega}(z)|^p\Big)(z),\quad{z\in\mathbb{D}}.
\end{aligned}
$$
Obviously, the case of $p=1$ is also satisfied.
Then Proposition \ref{BMOr} implies that there exists $\alpha> \alpha_0$ such that $\mathrm{BMO}_{\omega}^p\subseteq \mathcal{ BMO}_{\omega,\alpha}^p$. The proof is finished.
\end{proof}
Now we are ready to prove Theorem \ref{TB} which characterizes the boundedness of Toeplitz operator with $\mathrm{BMO}^p_{\omega}$-symbols via the Berezin-type transform.

\vspace{8pt}
\noindent \textbf{Proof of Theorem \ref{TB}.} The necessity is obvious, we only need to prove the sufficiency.

For $1<p<\infty$, assume $f\in \mathrm{BMO}^p_{\omega}$. Then by Propostion \ref{BMOeta}, there exists $\alpha_0=\alpha(\omega,p)$ such that $f\in \mathcal{ BMO}_{\omega,\alpha}^p$ for $\alpha>\alpha_0$.
The triangle inequality and H\"{o}lder's inequality imply
$$
\begin{aligned}
B_\omega^\alpha(|f|)(z)&=B_{\omega}^\alpha(|f|)(z)-|B_{\omega}^\alpha(f)(z)|+ |B_{\omega}^\alpha(f)(z)|\\
&\leq\mathcal{ MO}_{\omega,\alpha}^p(f)(z)+|B_{\omega}^\alpha(f)(z)|,\quad{z\in \mathbb{D}}.
\end{aligned}
$$
Moreover, for $z\in\mathbb{D}$, from \eqref{Dweight3} we deduce
$$
\begin{aligned}
B_{\omega}^\alpha(|f|)(z)
&\gtrsim \frac{(1-|z|)^{2\alpha+3}}{{\widehat{\omega}(z)}}\int_{D(z,r)}\frac{|f(\zeta)|}{|1-\bar{z}\zeta|^{2\alpha+4}}\omega(\zeta)\mathrm{d}A(\zeta)\\
&\asymp \frac{1}{\widehat{\omega}(z)(1-|z|)}\int_{D(z,r)}|f(\zeta)|\omega(\zeta)\mathrm{d}A(\zeta)\\
& = \frac{\mu(D(z,r))}{\widehat{\omega}(z)(1-|z|)},
\end{aligned}
$$
where $\mathrm{d}\mu(z)=|f(z)|\omega(z)\mathrm{d}A(z)$. By the assumption that $\omega\in \mathcal{D}$ and the boundedness of $B_\omega^\alpha(f)$ on $\mathbb{D}$, it follows that $\sup\limits_{z\in\mathbb{D}}\frac{\mu(D(z,r))}{\widehat{\omega}(z)(1-|z|)}<\infty,$ that is,
$\mu$ is a $p$-Carleson measure for $A_\omega^p$ (see
\cite[Theorem 2]{LR}). By \cite[Theorem 1]{PRS} and \cite[Theorem 1.1]{DGWW2}, the Toeplitz operator $T^\omega_{|f|}$ is bounded on $A_\omega^p$. Clearly, this implies that $T_f^\omega$ is also bounded on $A_\omega^p$.
\qed

\subsection{The proof of Theorem \ref{main1}}
Firstly we recall some notation used in \cite{YZ}. For $k\in\mathbb{N}\cup\{0\}$, let $r_k =$
$1 - {2}^{-k}$ and ${\theta }_{j} = \frac{j2\pi }{{2}^{k}}$ for $j = 0,1,\ldots ,{2}^{k}$. Define
$${S}_{kj} = \left\{  {r{e}^{i\theta } : {r}_{k} \leq  r < {r}_{k + 1},{\theta }_{j} \leq  \theta  < {\theta }_{j + 1}}\right\},$$
and
$${\widetilde{S}}_{kj} = \left\{  {r{e}^{i\theta } : {r}_{k - 1} \leq  r < {r}_{k + 2},{\theta }_{j - 1} \leq  \theta  < {\theta }_{j + 2}}\right\}.$$
Obviously, $S_{kj}\subset{\widetilde{S}}_{kj}$.
To prove Theorem \ref{main1}, we need the following lemmas.
\begin{lemma}\label{b1}
Let $\omega\in\widehat{\mathcal{D}}$ and $h\in H(\mathbb{D})$. Then for $k\in\mathbb{N}\cup\{0\}$ and $j=0,1,...,2^k$, we have
$$
|h(z)|\lesssim\frac{1}{\widehat{\omega}(r_k)(1-r_k)}\int_{\widetilde{S}_{kj}}|h(\zeta)|\widetilde{\omega}(\zeta)\mathrm{d}A(\zeta),
\quad{z\in \overline{S_{kj}}.}
$$
\end{lemma}
\begin{proof}
Let $B(z,\frac{1-r_k}{16})$ be the Euclidean disk centered at $z$ with radius $\frac{1-r_k}{16}$. For each $z\in\overline{S_{kj}}$, we have $B(z,\frac{1-r_k}{16})\subset\widetilde{S}_{kj}$. If $\zeta\in B(z,\frac{1-r_k}{16})$, then $r_{k-1}\leq |\zeta|<r_{k+2}$ and thus
\begin{equation}\label{eq3.9}
2\leq \frac{1-r_k}{1-|\zeta|}\leq 4.
\end{equation}
Moreover, since $\omega\in\widehat{\mathcal{D}}$, for $\zeta\in B(z,\frac{1-r_k}{16})$ we have
\begin{equation}\label{eq3.10}
\widehat{\omega}(\zeta)\leq \widehat{\omega}(r_{k-1}) \lesssim \widehat{\omega}\bigg(\frac{1+r_{k-1}}{2}\bigg)=\widehat{\omega}(r_{k}).
\end{equation}
On the other hand, an application of Lemma \ref{Dhat}(ii) shows
\begin{equation}\label{eq3.11}
\widehat{\omega}(r_{k})\leq\widehat{\omega}(r_{k-1})\lesssim \bigg(\frac{1-r_{k-1}}{1-|\zeta|}\bigg)^{\beta}\widehat{\omega}(\zeta)
<8^\beta\widehat{\omega}(\zeta),\quad \forall \, \zeta\in B\big(z,\frac{1-r_k}{16}\big).
\end{equation}
By the subharmonicity of $|h|$, \eqref{eq3.9}, \eqref{eq3.10} and \eqref{eq3.11}, we have
\begin{align*}
|h(z)|
&\leq\frac{1}{A(B(z,\frac{1-r_k}{16}))}\int_{B(z,\frac{1-r_k}{16})}|h(\zeta)|\mathrm{d}A(\zeta) \\
&\asymp\frac{(1-r_k)}{\widehat{\omega}(r_k)A(B(z,\frac{1-r_k}{16}))}\int_{B(z,\frac{1-r_k}{16})}|h(\zeta)|\frac{\widehat{\omega}(\zeta)}{1-|\zeta|}\mathrm{d}A(\zeta)\\
&\lesssim \frac{1}{\widehat{\omega}(r_k)(1-r_k)}\int_{\widetilde{S}_{kj}}|h(\zeta)|\widetilde{\omega}(\zeta)\mathrm{d}A(\zeta),\quad \forall\,z\in\overline{S_{kj}},
\end{align*}
which completes the proof.
\end{proof}

The next lemma comes from \cite{YZ}.
\begin{lemma}\label{b2}
For any $z\in\mathbb{D}$, we have
$$
\sum_{k=0}^\infty\sum_{j=0}^{2^k-1}\chi_{\widetilde{S}_{kj}}\leq 9,
$$
where $\chi_{\widetilde{S}_{kj}}$ is the characteristic function on the set $\widetilde{S}_{kj}$.
\end{lemma}

Based on some ideas in \cite{TV1}, we now give the proof of Theorem \ref{main1}.
	
\vspace{5pt}
\noindent \textbf{Proof of Theorem \ref{main1}(i).}
Let
\begin{equation}\label{epsilonk}
\epsilon_k=\sup\Big\{\big|{\bf M}(f,r,\theta,\tau)\big|:r_k\leq r \leq r_{k+1} ,\,\,\theta\in[0,2\pi],\,\, 0<\tau\leq1\Big\}.
\end{equation}
Let $C$ denote the supremum in \eqref{A1}. By the assumption \eqref{A1}, $C$ is finite. For each nonegative integer $k$ we have $\epsilon_k\leq C$. Let $1<p<\infty$ with $1/p+1/p'=1$. By Fubini's theorem and the duality, it suffices to show that
$$
|\langle T_f^\omega(g),h\rangle_{L_\omega^2}|=\left|\int_{\mathbb{D}} f(z) g(z) \overline{h(z)}\omega(z) \mathrm{d}A(z)\right| \lesssim \|g\|_{A_\omega^p}
\|h\|_{A_\omega^{p'}}
$$
for $g, h \in H^{\infty}.$ To do so, let $b$ denote the product of $g$ and $\bar{h}$. According to the definition of ${S}_{kj}$, $\mathbb D$ admits the decomposition $\mathbb{D}
= \mathop{\bigcup }\limits_{{k = 0}}^{\infty }\mathop{\bigcup }\limits_{{j = 0}}^{{{2}^{k} - 1}}{S}_{kj}$. As a result, we have
$$
\int_{\mathbb{D}}f(z)g(z)\overline{h(z)}\omega(z)\mathrm{d}A(z)=\sum_{k=0}^{\infty} \sum_{j=0}^{2^k-1} \int_{S_{k j}}f(z)b(z)\omega(z)\mathrm{d}A(z).
$$
Using integration by parts, we get
\begin{eqnarray}\label{Aa}
&&\int_{S_{k j}}f(z)b(z)\omega(z)\mathrm{d}A(z)\\
&=&  \int_{r_k}^{r_{k+1}}\int_{\theta_j}^{\theta_{j+1}} r\omega(r) f(r e^{i \theta}) \mathrm{d} \theta b(r e^{i \theta_{j+1}}) \mathrm{d}r
-\int_{r_k}^{r_{k+1}} \int_{\theta_j}^{\theta_{j+1}}\bigg[\int_{\theta_j}^\theta r\omega(r) f(r e^{i \eta}) \mathrm{d}\eta\bigg]\bigg[\frac{\partial}{\partial \theta} b(r
e^{i \theta})\bigg]\mathrm{d}\theta \mathrm{d}r.\nonumber
\end{eqnarray}

Using integration by parts for the variable $r$, we obtain

$$
\int_{r_k}^{r_{k+1}}  {\bigg[\int_{\theta_j}^{\theta_{j+1}} r\omega(r) f(r e^{i \theta}) \mathrm{d} \theta\bigg] b(r e^{i \theta_{j+1}}) \mathrm{d} r } \\
=I_{k j}^1-I_{k j}^2,
$$
where $$
I_{k j}^1=\bigg[\int_{r_k}^{r_{k+1}} \int_{\theta_j}^{\theta_{j+1}} r\omega(r) f(r e^{i \eta}) \mathrm{d} \eta \mathrm{d} r\bigg] b(r_{k+1} e^{i
\theta_{j+1}}),
$$ and
$$
I_{k j}^2=\int_{r_k}^{r_{k+1}}\bigg[\int_{r_k}^r \int_{\theta_j}^{\theta_{j+1}} t\omega(t) f(t e^{i \eta}) \mathrm{d} \eta \mathrm{d}
t\bigg]\bigg[\frac{\partial}{\partial r} b(r e^{i \theta_{j+1}})\bigg] \mathrm{d} r.
$$
Applying Fubini's theorem and integration by parts for the variable $r$ gives

$$
\begin{aligned}
& \int_{r_k}^{r_{k+1}} \int_{\theta_j}^{\theta_{j+1}}\bigg[\int_{\theta_j}^\theta r\omega(r) f(r e^{i \eta}) \mathrm{d}
\eta\bigg]\bigg[\frac{\partial}{\partial \theta} b(r e^{i \theta})\bigg] \mathrm{d} \theta \mathrm{d} r \\
=& \int_{\theta_j}^{\theta_{j+1}} \int_{r_k}^{r_{k+1}}\bigg[\int_{\theta_j}^\theta r\omega(r) f(r e^{i \eta}) \mathrm{d}
\eta\bigg]\bigg[\frac{\partial}{\partial \theta} b(r e^{i \theta})\bigg] \mathrm{d} r \mathrm{d} \theta \\
=& \int_{\theta_j}^{\theta_{j+1}}\bigg[\int_{r_k}^{r_{k+1}} \int_{\theta_j}^\theta t\omega(t) f\left(t e^{i \eta}\right) \mathrm{d} \eta \mathrm{d}
t\bigg]\bigg[\frac{\partial}{\partial \theta} b(r_{k+1} e^{i \theta})\bigg] \mathrm{d} \theta \\
&-\int_{\theta_j}^{\theta_{j+1}} \int_{r_k}^{r_{k+1}}\bigg[\int_{r_k}^r \int_{\theta_j}^\theta t\omega(t) f(t e^{i \eta}) \mathrm{d} \eta \mathrm{d}
t\bigg]\bigg[\frac{\partial^2}{\partial r \partial \theta} b(r e^{i \theta})\bigg] \mathrm{d} r \mathrm{d} \theta.
\end{aligned}
$$
Then the integral in \eqref{Aa} can be rewritten as
$$
\int_{S_{k j}} f(z) b(z) \mathrm{d} A(z)=I_{k j}^1-I_{k j}^2-I_{k j}^3+I_{k j}^4,
$$
where
$$
 I_{k j}^3=\int_{\theta_j}^{\theta_{j+1}}\bigg[\int_{r_k}^{r_{k+1}} \int_{\theta_j}^\theta t\omega(t) f\left(t e^{i \eta}\right) \mathrm{d} \eta \mathrm{d}
t\bigg]\bigg[\frac{\partial}{\partial \theta} b(r_{k+1} e^{i \theta})\bigg] \mathrm{d} \theta,
$$
and
$$
 I_{k j}^4=\int_{\theta_j}^{\theta_{j+1}} \int_{r_k}^{r_{k+1}}\bigg[\int_{r_k}^r \int_{\theta_j}^\theta t\omega(t) f(t e^{i \eta}) \mathrm{d} \eta
\mathrm{d} t\bigg]\bigg[\frac{\partial^2}{\partial r \partial \theta} b(r e^{i \theta})\bigg] \mathrm{d} r \mathrm{d} \theta.
$$
In the following, we proceed to estimate each of them.
For $I_{k j}^1$, we have
$$
\begin{aligned}
\left|I_{k j}^1\right|
&\leq \bigg(\bigg|\int_{r_k}^1 \int_{\theta_j}^{\theta_{j+1}} r\omega(r) f(r e^{i \eta}) \mathrm{d} \eta \mathrm{d}
r\bigg|
 +\bigg|\int_{r_{k+1}}^1 \int_{\theta_j}^{\theta_{j+1}} r\omega(r) f(r e^{i \eta}) \mathrm{d} \eta \mathrm{d} r\bigg|\bigg)\left|b(r_{k+1} e^{i
\theta_{j+1}})\right|\\
&=\bigg(\bigg|\int_{E_{jk}^1}f(\zeta)\omega(\zeta)\mathrm{d}A(\zeta)\bigg|+\bigg|\int_{E_{jk}^2}f(\zeta)\omega(\zeta)\mathrm{d}A(\zeta)\bigg|\bigg)\left|b(r_{k+1} e^{i
	\theta_{j+1}})\right|,
\end{aligned}
$$
where
$E_{jk}^1=S_{1-{r_k}}^1(e^{i \frac{\theta_{j+1}+\theta_j}{2}})$,
$E_{jk}^2=S_{1-r_{k+1}}^1(e^{i \eta_{0}}) \cup S_{1-r_{k+1}}^1(e^{i \eta_{1}})$
and $\eta_{0}, \eta_{1} \in[\theta_j,
\theta_{j+1}]$ are chosen such that $S_{1-r_{k+1}}^1(e^{i\eta_{0}})$ and $S_{1-r_{k+1}}^1(e^{i \eta_{1}})$ are disjoint. Then the above inequality together with the definition of $\epsilon_k$ gives
\begin{eqnarray}\label{A2}
	\left|I_{k j}^1\right|
 &\leq&\epsilon_k\big[\widehat{\omega}(r_k)(1-r_k)+2\widehat{\omega}(r_{k+1})(1-r_{k+1})\big]\left|b(r_{k+1} e^{i\theta_{j+1}})\right|\nonumber\\
 &\asymp&\epsilon_k\widehat{\omega}(r_k)(1-r_k)\left|b(r_{k+1} e^{i\theta_{j+1}})\right|.
\end{eqnarray}
Since $b=g \bar{h}$ and $g, h \in H^{\infty}$, it follows from Lemma \ref{b1} that
\begin{equation}\label{A3}
\left|b\left(r_{k+1} e^{i \theta_{j+1}}\right)\right| \lesssim \frac{1}{\widehat{\omega}(r_k)(1-r_k)} \int_{\tilde{S}_{k
j}}|b(z)|\widetilde{\omega}(z) \mathrm{d} A(z).
\end{equation}
Combining \eqref{A2} with \eqref{A3}, we obtain
$$
\left|I_{k j}^1\right| \lesssim \epsilon_k \int_{\tilde{S}_{k j}}|b(z)|\widetilde{\omega}(z) \mathrm{d} A(z).
$$
Next we consider $I_{k j}^2$. Rewrite $I_{kj}^2$ as
$$
I_{k j}^2=\int_{r_k}^{r_{k+1}}\left[\int_{r_k}^1 \int_{\theta_j}^{\theta_{j+1}} t\omega(t) f(t e^{i \eta}) \mathrm{d} \eta \mathrm{d} t-\int_r^1
\int_{\theta_j}^{\theta_{j+1}} t\omega(t) f(t e^{i \eta}) \mathrm{d} \eta \mathrm{d} t\right] \frac{\partial}{\partial r}\left[b(r e^{i \theta_{j+1}})\right]
\mathrm{d} r.
$$
Note that in the above second inner integral
\begin{equation}\label{A4}
\int_r^1 \int_{\theta_j}^{\theta_{j+1}}t\omega(t) f(t e^{i \eta}) \mathrm{d} \eta \mathrm{d} t,
\end{equation}
we have $r_k \leq r \leq r_{k+1}$ and $\theta_{j+1}-\theta_j=2 \pi / 2^k$. Thus the domain of this integral, denoted by $E_{jk}^3(r)$, can be written as the union of the disjoint sets
$S_{1-r}^{\tau_{0,r}}(e^{i \eta_{0,r}})$ and $S_{1-r}^{\tau_{1,r}}(e^{i \eta_{1,r}})$ for some $0<\tau_{0,r},\tau_{1,r}\leq 1$ and
$\eta_{0,r},\eta_{1,r} \in[\theta_j, \theta_{j+1}]$.
Consequently, the same reasoning as in \eqref{A2} shows that
$$
\begin{aligned}
&\left|\int_{r_k}^1 \int_{\theta_j}^{\theta_{j+1}}t\omega(t) f(t e^{i \eta}) \mathrm{d} \eta \mathrm{d} t\right|
+\left|\int_r^1 \int_{\theta_j}^{\theta_{j+1}}t\omega(t) f(t e^{i \eta}) \mathrm{d} \eta \mathrm{d} t\right|\\
=&\left|\int_{E_{jk}^1} f(z)\omega(z) \mathrm{d} A(z)\right|
+\left|\int_{E_{jk}^3} f(z)\omega(z) \mathrm{d} A(z)\right| \\
\leq &\epsilon_k \widehat{\omega}(r_k)(1-r_k)+ 2 \epsilon_k \widehat{\omega}(r)(1-r).
\end{aligned}
$$
Since $\widehat{\omega}(r)\asymp \widehat{\omega}(r_k)$ for $r\in[r_k,r_{k+1}]$, we obtain
$$
\left|I_{k j}^2\right| \lesssim \epsilon_k \widehat{\omega}(r_k)(1-r_k) \int_{r_k}^{r_{k+1}}\left|\frac{\partial}{\partial r}\left[b(r e^{i
\theta_{j+1}})\right]\right| \mathrm{d} r.
$$
Moreover, since $b(z)=g(z) \overline{h(z)}$, we get
\begin{equation}\label{A5}
\left|\frac{\partial}{\partial r}\left[b(r e^{i \theta_{j+1}})\right]\right|
\leq\left|g^{\prime}(r e^{i \theta_{j+1}}) h(r e^{i \theta_{j+1}})\right|+\left|g(r
e^{i \theta_{j+1}}) h^{\prime}(r e^{i \theta_{j+1}})\right|.
\end{equation}
For any $z \in \widetilde{S}_{k j}$, we have
\begin{equation}\label{skj}
1-|z|^2 \asymp r_{k+1}-r_k \asymp \theta_{j+1}-\theta_j.
\end{equation}
Lemma \ref{b1} together with \eqref{A5} and \eqref{skj} yields
$$
\left|I_{k j}^2\right| \lesssim \epsilon_k \int_{\tilde{S}_{k j}} F(z) \widetilde{\omega}(z)\mathrm{d} A(z),
$$
where $F(z)=|g^{\prime}(z)|(1-|z|^2) |h(z)|+|g(z)| |h^{\prime}(z)|(1-|z|^2)$.

Now we turn to estimate $I_{kj}^3$. Rewrite $I_{k j}^3$ as
$$
\int_{\theta_j}^{\theta_{j+1}}\bigg[\int_{r_k}^{1} \int_{\theta_j}^\theta t\omega(t) f(t e^{i \eta}) \mathrm{d} \eta \mathrm{d} t-\int_{r_{k+1}}^{1}
\int_{\theta_j}^\theta t\omega(t) f(t e^{i \eta}) \mathrm{d} \eta \mathrm{d} t\bigg]\bigg[\frac{\partial}{\partial \theta} b(r_{k+1} e^{i
\theta})\bigg] \mathrm{d} \theta,
$$
and rewrite the above first inner integral
$$
\int_{r_k}^1 \int_{\theta_j}^\theta t\omega(t) f(t e^{i \eta}) \mathrm{d} \eta \mathrm{d} t
=\int_{S_{1-r_k}^{\beta_{\theta}}(e^{i\frac{\theta+\theta_j}{2}})} f(\zeta)\omega(\zeta)\mathrm{d}A(\zeta),
$$
 where
$\beta_{\theta}=\frac{2^{k-1}(\theta-\theta_j)}{\pi}$.
It follows from the definition of $\epsilon_k$ that
\begin{equation}\label{A6}
	\left|\int_{r_k}^1 \int_{\theta_j}^\theta t\omega(t) f\left(t e^{i \eta}\right) \mathrm{d} \eta\mathrm{d} t\right|
	\leq \epsilon_k \widehat{\omega}(r_k)(1-r_k).
\end{equation}
Note that $\theta_j\leq\theta\leq\theta_{j+1}$ in the second inner integral, we estimate $$\int_{r_{k+1}}^{1} \int_{\theta_j}^\theta t\omega(t) f(t e^{i \eta}) \mathrm{d} \eta
\mathrm{d} t$$
in following two cases: $2\pi/2^{k+1}<\theta-\theta_j\leq2\pi/2^k$ and $0\leq\theta-\theta_j\leq2\pi/2^{k+1}$.

If $2\pi/2^{k+1}<\theta-\theta_j\leq2\pi/2^k$, then the domain of the integral can be seen as $E_{jk}^4(\theta)$, where $E_{jk}^4(\theta)$ is the union of disjoint
$S_{1-r_{k+1}}^{\beta_{0,\theta}}(e^{i \eta_0})$ and $S_{1-r_{k+1}}^{\beta_{1,\theta}}(e^{i \eta_1})$ for some $0<\beta_{0,\theta}$, $\beta_{1,\theta} \leq 1$ and
$\eta_0$, $\eta_1\in[\theta, \theta_{j+1}]$.
In this case, we have
$$
\left|\int_{r_{k+1}}^{1} \int_{\theta_j}^\theta t\omega(t) f(t e^{i \eta}) \mathrm{d} \eta
\mathrm{d} t\right|
=\left|\int_{E_{jk}^4}f(z)\omega(z)\mathrm{d}A(z)\right|\lesssim \epsilon_k\widehat{\omega}(r_{k})(1-r_{k}).
$$
If $0<\theta-\theta_j\leq2\pi/2^{k+1}$, then
$$
\left|\int_{r_{k+1}}^{1} \int_{\theta_j}^\theta t\omega(t) f(t e^{i \eta}) \mathrm{d} \eta \mathrm{d} t\right|
=\left|\int_{S_{1-r_{k+1}}^{\gamma_{\theta}}(e^{i \frac{\theta+\theta_j}{2}})}f(z)\omega(z)\mathrm{d}A(z)\right|
\lesssim \epsilon_k\widehat{\omega}(r_{k})(1-r_{k}),
$$
where $\gamma_{\theta}=\frac{2^{k}(\theta-\theta_j)}{\pi}$.

Moreover, we also have
$$
\left|\frac{\partial}{\partial \theta} b(r_{k+1} e^{i \theta})\right|\leq\left|g^{\prime}(r_{k+1} e^{i \theta}) h(r_{k+1} e^{i \theta})\right|+\left|g(r_{k+1} e^{i \theta}) h^{\prime}(r_{k+1} e^{i \theta})\right|.
$$
Similarly, the above inequality together with Lemma \ref{b1} and \eqref{skj} shows that
$$
\left|I_{k j}^3\right|
\lesssim \epsilon_k \int_{\tilde{S}_{k j}}F(z)\widetilde{\omega}(z)\mathrm{d} A(z).
$$
For $I_{k j}^4$, we write
$$
\begin{aligned}
I_{k j}^4= & \int_{\theta_j}^{\theta_{j+1}} \int_{r_k}^{r_{k+1}}\bigg[\int_{r_k}^1 \int_{\theta_j}^\theta t\omega(t) f(t e^{i \eta}) \mathrm{d} \eta
\mathrm{d} t-\int_r^1 \int_{\theta_j}^\theta t\omega(t) f(t e^{i \eta}) \mathrm{d} \eta \mathrm{d} t\bigg] \\
& \times\bigg[\frac{\partial^2}{\partial r \partial \theta} b\left(r e^{i \theta}\right)\bigg] \mathrm{d} r \mathrm{d} \theta.
\end{aligned}
$$
Similarly to the estimation of \eqref{A4}, we get
\begin{equation}\label{A7}
\left|\int_{r}^1 \int_{\theta_j}^\theta t\omega(t) f\left(t e^{i \eta}\right) \mathrm{d} \eta\mathrm{d} t\right| \lesssim \epsilon_k
\widehat{\omega}(r)(1-r)\asymp \epsilon_k \widehat{\omega}(r_k)(1-r_k) .
\end{equation}
Since $b(z)=g(z) \overline{h(z)}$, it follows that
\begin{eqnarray}\label{A8}
\left|\frac{\partial^2}{\partial r\partial \theta} b(r e^{i \theta})\right|
&\leq & \left|g^{\prime}(r e^{i \theta}) h(r e^{i \theta})\right|+\left|g(r e^{i \theta}) h^{\prime}(r e^{i \theta})\right|
 +2\left|g^{\prime}(r e^{i \theta}) h^{\prime}(r e^{i \theta})\right| \nonumber\\
 &+&\left|g^{\prime \prime}(r e^{i \theta}) h(r e^{i \theta})\right|
 +\left|g(r e^{i \theta}) h^{\prime \prime}(r e^{i \theta})\right|.
\end{eqnarray}
Then Lemma \ref{b1} combined with \eqref{skj}, \eqref{A6}, \eqref{A7} and
\eqref{A8} implies
$$
\left|I_{k j}^4\right| \lesssim \epsilon_k \int_{\tilde{S}_{k j}}\big(F(z)+G(z)\big)\widetilde{\omega}(z)\mathrm{d} A(z),
$$
where $G(z)=\left|g^{\prime}(z) h^{\prime}(z)(1-|z|^2)^2\right|+\left|g^{\prime \prime}(z)(1-|z|^2)^2 h(z)\right| +\left|g(z) h^{\prime
\prime}(z)(1-|z|^2)^2\right|$.
Combining the estimation for $I_{k j}^1, I_{k j}^2, I_{k j}^3$ and $I_{k j}^4$ together, we deduce
\begin{equation}\label{A9}
\left|\int_{S_{k j}} f(z) b(z) \omega(z)\mathrm{d} A(z)\right|
\lesssim \epsilon_k \int_{\tilde{S}_{k j}} \bigg(|g(z)||h(z)|+F(z)+G(z)\bigg) \widetilde{\omega}(z) \mathrm{d} A(z).
\end{equation}
By \eqref{RD}, \eqref{A9}, Lemma \ref{b2} and H\"{o}lder's inequality,  we get
$$
\begin{aligned}
&\left|\int_{\mathbb{D}} f(z) g(z) \overline{h(z)} \omega(z)\mathrm{d} A(z)\right|\\
&\lesssim \sum_{k=0}^{\infty} \sum_{j=0}^{2^k-1} \int_{\tilde{S}_{k j}} \bigg(|g(z)||h(z)|+F(z)+G(z)\bigg)\widetilde{\omega}(z) \mathrm{d} A(z)
\\
&\lesssim \int_{\mathbb{D}}|g(z)||h(z)|\omega(z) \mathrm{d} A(z)
+\int_{\mathbb{D}}\big(F(z)+G(z)\big)\widetilde{\omega}(z) \mathrm{d} A(z) \\
& \leq \|g\|_{A_\omega^p}\|h\|_{A_\omega^{p'}}+\int_{\mathbb{D}} F(z)\widetilde{\omega}(z) \mathrm{d}A(z)+\int_{\mathbb{D}} G(z)\widetilde{\omega}(z)
\mathrm{d}A(z).
\end{aligned}
$$
Moreover, since $\omega\in\mathcal{D}$, it follows from Lemma \ref{DR} that $\widetilde{\omega}\in \mathcal{R}$. Then H\"{o}lder's inequality combined with \eqref{Littlewood} and \eqref{RD} shows that
\begin{eqnarray}\label{A10}
\int_{\mathbb{D}} F(z)\widetilde{\omega}(z) \mathrm{d}A(z)
&\leq& \left(\int_{\mathbb{D}} |g^{\prime}(z)|^p(1-|z|)^p \widetilde{\omega}(z)\mathrm{d}A(z)\right)^{1/p} \|h\|_{A_{\widetilde{\omega}}^{p'}}\nonumber\\
&+&\|g\|_{A_{\widetilde{\omega}}^{p}}\left(\int_{\mathbb{D}} |h^{\prime}(z)|^{p'}(1-|z|)^{p'}\widetilde{\omega}(z)\mathrm{d}A(z)\right)^{1/p'}\nonumber\\
&\lesssim &\|g\|_{A_\omega^p}\|h\|_{A_\omega^{p'}}.
\end{eqnarray}
We can also get $\int_{\mathbb{D}} G(z)\widetilde{\omega}(z)\mathrm{d}A(z)\lesssim \|g\|_{A_\omega^p}\|h\|_{A_\omega^{p'}}$ in this way.
It follows that
$$
\left|\int_{\mathbb{D}} f(z) g(z) \overline{h(z)} \omega(z)\mathrm{d} A(z)\right|\lesssim \|g\|_{A_\omega^p}\|h\|_{A_\omega^{p'}}, \quad{g,h\in H^\infty}.
$$
Since $H^\infty$ is dense in $A_\omega^p$, a standard density argument yields $\|T_f^\omega\|_{A_\omega^p}\lesssim
1$, which completes the proof.
\hfill$\Box$

\hspace{10pt}

Next, we prove the second part of Theorem \ref{main1}.

\vspace{5pt}
\noindent \textbf{Proof of Theorem \ref{main1}(ii).}
Suppose that $f$ satisfies \eqref{B0}. Then it follows from the definition of $\epsilon_k$ in \eqref{epsilonk} that for any $\varepsilon>0$, there is a positive integer $m$ such that $\epsilon_k<\varepsilon$ whenever $k
\geq m$.
Let $\{g_n\}_{n=1}^\infty$ be any sequence in $A_\omega^p$ such that $g_n$ converges to 0 uniformly on any compact subset of $\mathbb{D}$,
and $\left\|g_n\right\|_{A_\omega^p} \leq K$ for some positive constant $K$. To show that $T_f^\omega$ is compact on $A_\omega^p$, we only need to show that
$\big\|T_f^\omega g_n\big\|_{A_\omega^p} \rightarrow 0$ as $n \rightarrow \infty$.

Let $h \in H^{\infty}$. Similarly to the proof of Theorem \ref{main1}(i),
we obtain
\begin{equation}\label{A11}
\left|\int_{\mathbb{D}} f(z) g_n(z) \overline{h(z)} \omega(z)\mathrm{d} A(z)\right| \lesssim \sum_{k=0}^{\infty} \sum_{j=0}^{2^k-1} \epsilon_k
\int_{\tilde{S}_{k j}} F_n(z) \widetilde{\omega}(z) \mathrm{d} A(z),
\end{equation}
where
\begin{align*}
F_n(z)= & \left|g_n(z) h(z)\right|+\left|g_n^{\prime}(z)\left(1-|z|^2\right) h(z)\right|
 +\left|g_n(z) h^{\prime}(z)\left(1-|z|^2\right)\right|\\
 &+\left|g_n^{\prime}(z) h^{\prime}(z)\left(1-|z|^2\right)^2\right|+\left|g_n^{\prime \prime}(z)\left(1-|z|^2\right)^2 h(z)\right|\\
 &+\left|g_n(z) h^{\prime \prime}(z)\left(1-|z|^2\right)^2\right|.
\end{align*}
We split the summation in \eqref{A11} into the following two parts:
$$
\begin{aligned}
\left|\int_{\mathbb{D}} f(z) g_n(z) \overline{h(z)}\omega(z) \mathrm{d} A(z)\right|
\lesssim I_{m,n}^1+I_{m,n}^2,
\end{aligned}
$$
where
$$
I_{m,n}^1=\sum_{k=0}^{m-1} \sum_{j=0}^{2^k-1} \epsilon_k \int_{\tilde{S}_{k j}} F_n(z) \widetilde{\omega}(z) \mathrm{d} A(z),\quad
I_{m,n}^2=\sum_{k=m}^{\infty} \sum_{j=0}^{2^k-1} \epsilon_k \int_{\tilde{S}_{k j}} F_n(z)  \widetilde{\omega}(z)\mathrm{d} A(z).
$$
Since $g_n$ is holomorphic, by Cauchy's integral formula the sequences $g_n^{\prime}$ and $g_n^{\prime \prime}$ also converge to 0
uniformly on any compact subset of $\mathbb{D}$.
Then H\"{o}lder's inequality together with \eqref{Littlewood} and \eqref{RD} implies that there is a positive integer $N$ such that

$$
\int_{\{z \in \mathbb{D}:|z| \leq 1-\frac{1}{2^{m+1}}\}} F_n(z) \widetilde{\omega}(z)\mathrm{d} A(z)
\lesssim \widetilde{\omega}(\mathbb{D})^{1/p}\|h\|_{A_\omega^{p'}} \varepsilon, \quad \forall n>N .
$$
Note that
$$
\bigcup_{k=0}^{m-1} \bigcup_{j=0}^{2^k-1} \widetilde{S}_{k j} \subset\left\{z \in \mathbb{D}:|z| \leq 1-\frac{1}{2^{m+1}}\right\}.
$$
It follows that
$$
I_{m,n}^1
\lesssim \max \limits_{\left\{0 \leq k \leq m-1\right\}} \epsilon_k\int_{\{z \in \mathbb{D}:|z| \leq 1-\frac{1}{2^{m+1}}\}} F_n(z) \widetilde{\omega}(z)\mathrm{d} A(z)
\lesssim\|h\|_{A_\omega^{p'}} \varepsilon, \quad \forall n>N .
$$
Moreover, since $\epsilon_k<\varepsilon$ whenever $k \geq m$, in the same way as \eqref{A10}, we deduce
$$
I_{m,n}^2  <\varepsilon \sum_{k=m}^{\infty} \sum_{j=0}^{2^k-1} \int_{\tilde{S}_{k j}} F_n(z) \widetilde{\omega}(z) \mathrm{d} A(z)
 \lesssim \varepsilon\left\|g_n\right\|_{A_\omega^p}\|h\|_{A_\omega^{p'}} \leq K\|h\|_{A_\omega^{p'}} \varepsilon.
$$
Combining the estimations of $I_{m,n}^1$ and $I_{m,n}^2$, we conclude that for any $\varepsilon>0$ there is a positive integer $N$ such that
$$
\left|\int_{\mathbb{D}} f(z) g_n(z) \overline{h(z)}\omega(z) \mathrm{d} A(z)\right| \lesssim\|h\|_{A_\omega^{p'}} \varepsilon, \quad \forall n>N.
$$
Thus $\big\|T_f^\omega g_n\big\|_{A_\omega^p}$ converges to 0 as $n \rightarrow \infty$, which completes the proof of this theorem.
\qed

\vspace{.3cm}
In \cite{TV1}, Taskinen and Virtanen gave a sufficient condition via the half Carleson square for the boundedness of Toeplitz operator with locally integrable symbols. Taskinen and Virtanen's condition can also be extended to the Bergman space with $\mathcal D$ weight. Indeed, using a similar idea in \cite[Theorem 3.2]{YZ}, we can show that the condition \eqref{A1} in Theorem \ref{main1} is equivalent to
$$
C_f:=\mathlarger{\mathlarger{\sup}}_{\substack{m \in \mathbb{N}^{+},\mu=0,1, \ldots, 2^m-1\\	 1-\frac{1}{2^{m-1}}\leq \rho <1-\frac{1}{2^m}\\ \frac{2\pi\mu}{2^{m}}\leq \theta <\frac{2\pi(\mu+1)}{2^{m}}}}
\frac{2^{m}}{\widehat{\omega}(1-\frac{1}{2^m})}\left|\int_{1-\frac{1}{2^{m-1}}}^{\rho}
\int_{\frac{2\pi\mu}{2^m}}^\theta f(re^{i\eta})r\omega(r)\mathrm{d}\eta \mathrm{d}r\right|<\infty.
$$
Combining this and Theorem \ref{main1} gives the following theorem.

\begin{theorem}
Let $\omega\in\mathcal{D}$ and $f \in L_{\mathrm{loc}}^1(\mathbb{D},\omega\mathrm{d} A)$. If $C_f<\infty$, then $T_f: A_\omega^p\rightarrow
A_\omega^p$ is bounded for $1<p<\infty$.
\end{theorem}

\subsection{The proof of Theorem \ref{main2}}
In this section, we prove Theorem \ref{main2} which gives a sufficient condition for the Toeplitz operators with locally integrable symbols to be bounded (or compact) on $A_\omega^1$ with $\mathcal{D}$ weights. The main strategy is similar to that in Theorem \ref{main1}.

\vspace{5pt}
\noindent \textbf{Proof of Theorem \ref{main2}.}
Let
\[
\delta_k=\sup\Big\{|\mathscr{M}(f,r,\theta,\tau)|:r_k \leq r \leq r_{k+1},\,\,\theta\in[0,2\pi],\,\, 0<\tau\leq1\Big\}.
\]
By \eqref{E0}, there exists a constant $C=C(\omega,f)>0$ such that $\delta_k\leq C$.

Since $\mathbb{D} = \mathop{\bigcup }\limits_{{k = 0}}^{\infty }\mathop{\bigcup }\limits_{{j = 0}}^{{{2}^{k} - 1}}{S}_{kj}$, for $g\in A_\omega^1$ we have
\begin{equation}\label{E1}
T_f^\omega(g)(z)
=\sum_{k=0}^{\infty} \sum_{j=0}^{2^k-1} \int_{S_{k j}}f(\zeta)g(\zeta)\overline{B_z^\omega(\zeta)}\omega(\zeta)\mathrm{d}A(\zeta),\quad{z\in\mathbb{D}}.
\end{equation}
Similarly to the proof of Theorem \ref{main1}, we apply integration by parts to get
\begin{align*}
&\sum_{k=0}^{\infty}\sum_{j=0}^{2^k-1} \int_{S_{k j}}f(\zeta)g(\zeta)\overline{B_z^\omega(\zeta)}\omega(\zeta)\mathrm{d}A(\zeta)\\
=&\bigg[\int_{r_k}^{r_{k+1}} \int_{\theta_j}^{\theta_{j+1}} r\omega(r) f(r e^{i \eta}) \mathrm{d} \eta \mathrm{d} r\bigg] g(r_{k+1} e^{i
\theta_{j+1}})\overline{B_z^\omega (r_{k+1} e^{i
\theta_{j+1}})} \\
&-\int_{r_k}^{r_{k+1}}\bigg[\int_{r_k}^r \int_{\theta_j}^{\theta_{j+1}} t\omega(t) f(t e^{i \eta}) \mathrm{d} \eta \mathrm{d}
t\bigg]\frac{\partial}{\partial r} \bigg(g(r e^{i \theta_{j+1}})\overline{B_z^\omega(r e^{i \theta_{j+1}})}\bigg) \mathrm{d} r\\
 &+\int_{\theta_j}^{\theta_{j+1}}\bigg[\int_{r_k}^{r_{k+1}} \int_{\theta_j}^\theta t\omega(t) f\left(t e^{i \eta}\right) \mathrm{d} \eta \mathrm{d}
t\bigg]\frac{\partial}{\partial \theta} \bigg(g(r_{k+1} e^{i \theta})\overline{B_z^\omega(r_{k+1} e^{i \theta})}\bigg) \mathrm{d} \theta, \\
&-\int_{\theta_j}^{\theta_{j+1}} \int_{r_k}^{r_{k+1}}\bigg[\int_{r_k}^r \int_{\theta_j}^\theta t\omega(t) f(t e^{i \eta}) \mathrm{d} \eta
\mathrm{d} t\bigg]\frac{\partial^2}{\partial r \partial \theta} \bigg(g(r e^{i \theta})\overline{B_z^\omega(r e^{i \theta})}\bigg) \mathrm{d} r \mathrm{d} \theta\\
=&L_{k j}^1-L_{k j}^2+L_{k j}^3-L_{k j}^4.
\end{align*}
We first consider $L_{k j}^1$. By an argument similar to that in Theorem \ref{main1}, we deduce from Lemma \ref{b1} that
\begin{align*}
\left|L_{k j}^1\right|
&\lesssim\frac{\delta_k}{\log{\frac{e}{1-r_k}}} \int_{\tilde{S}_{k j}}|g(\zeta)||B_z^\omega(\zeta)|\widetilde{\omega}(\zeta) \mathrm{d} A(\zeta)\\
&\lesssim  \int_{\tilde{S}_{k j}}\frac{|g(\zeta)|}{\log{\frac{e}{1-|\zeta|}}}|B_z^\omega(\zeta)|\widetilde{\omega}(\zeta) \mathrm{d}
A(\zeta),\quad{z\in\mathbb{D}}.
\end{align*}
Note that $r_{k+1}-r_k \asymp \theta_{j+1}-\theta_j \asymp 1-|\zeta|$ for $\zeta \in \widetilde{S}_{kj}$. In a way similar to the proof of Theorem \ref{main1}(i), we can also get the estimate of $L_{k j}^2, L_{k j}^3$ and $L_{k j}^4$. It follows from \eqref{E1} and Lemma \ref{b2} that
\begin{equation}\label{E2}
\left|\int_{\mathbb{D}}f(\zeta)g(\zeta)\overline{B_z^\omega(\zeta)}\omega(\zeta)\mathrm{d}A(\zeta)\right|
\lesssim \int_{\mathbb{D}}\frac{G_z^\omega(\zeta)}{\log{\frac{e}{1-|\zeta|}}}\widetilde{\omega}(\zeta)\mathrm{d}A(\zeta),
\end{equation}
where
\begin{align*}
G_z^\omega(\zeta)= & |g(\zeta) B_z^\omega(\zeta)|+\left|g^{\prime}(\zeta)(1-|\zeta|) B_z^\omega(\zeta)\right|+\left|g(\zeta)(1-|\zeta|) ({B_z^\omega})^{\prime}(\zeta)\right|\\
&+\left|g^{\prime}(\zeta) (1-|\zeta|)^2({B_z^\omega})^{\prime}(\zeta)\right|+\left|g^{\prime \prime}(\zeta)\left(1-|\zeta|\right)^2 B_z^\omega(\zeta)\right|\\
&+\left|g(\zeta) (1-|\zeta|)^2({B_z^\omega})^{\prime
\prime}(\zeta)\right|,\quad{z\in\mathbb{D}}.
\end{align*}
For simplicity, we only give the estimation of the former two integrals on the right hand side of \eqref{E2}.
Fubini's theorem together with Lemma \ref{Ne1} and \eqref{RD} gives
$$
\begin{aligned}
\left\|\int_{\mathbb{D}} \frac{|g(\zeta)|}{\log{\frac{e}{1-|\zeta|}}} |B_z^\omega(\zeta)|\widetilde{\omega}(\zeta) \mathrm{d} A(\zeta)\right\|_{L_\omega^1}
 =\int_{\mathbb{D}} \frac{|g(\zeta)|}{\log{\frac{e}{1-|\zeta|}}}\int_{\mathbb{D}}|B_\zeta^\omega(z)|\omega(z)\mathrm{d}
A(z)\widetilde{\omega}(\zeta)\mathrm{d} A(\zeta)
\asymp \|g\|_{A_\omega^1}.
\end{aligned}
$$
Moreover, the same reasoning combined with Lemma \ref{lemE} implies
\begin{align*}
\left\|\int_{\mathbb{D}}\frac{(1-|\zeta|)}{\log{\frac{e}{1-|\zeta|}}}|g^{\prime}(\zeta)B_z^\omega(\zeta)| \widetilde{\omega}(\zeta)\mathrm{d}A(\zeta)\right\|_{L_\omega^1}
&=\int_{\mathbb{D}}|g^{\prime}(\zeta)|\frac{(1-|\zeta|)}{\log{\frac{e}{1-|\zeta|}}}\int_{\mathbb{D}}|B_{\zeta}^\omega(z)|
\omega(z)\mathrm{d}A(z)\widetilde{\omega}(\zeta)\mathrm{d}A(\zeta)\\
&\asymp \int_{\mathbb{D}}|g^{\prime}(\zeta)|(1-|\zeta|))\widetilde{\omega}(\zeta)\mathrm{d}A(\zeta)\\
&\lesssim \|g\|_{A_{\widetilde{\omega}}^1}\asymp \|g\|_{A_\omega^1}.
\end{align*}
Similarly, we have
$$\left\|\int_{\mathbb{D}}\frac{(1-|\zeta|)}{\log{\frac{e}{1-|\zeta|}}}|g(\zeta)(B_z^\omega)^{\prime}(\zeta)| \widetilde{\omega}(\zeta)\mathrm{d}A(\zeta)\right\|_{L_\omega^1} \lesssim \|g\|_{A_\omega^1},$$	
$$\left\|\int_{\mathbb{D}}\frac{(1-|\zeta|)^2}{\log{\frac{e}{1-|\zeta|}}}|g^{\prime}(\zeta)(B_z^\omega)^{\prime}(\zeta)| \widetilde{\omega}(\zeta)\mathrm{d}A(\zeta)\right\|_{L_\omega^1}\lesssim  \|g\|_{A_\omega^1},
$$	
and
$$ \left\|\int_{\mathbb{D}}\frac{(1-|\zeta|)^2}{\log{\frac{e}{1-|\zeta|}}}|g(\zeta)(B_z^\omega)^{\prime\prime}(\zeta)| \widetilde{\omega}(\zeta)\mathrm{d}A(\zeta)\right\|_{L_\omega^1}\lesssim  \|g\|_{A_\omega^1}.
$$
As a consequence, we have
 $$\left\|\int_{\mathbb{D}}\frac{G_z^\omega(\zeta)}{\log{\frac{e}{1-|\zeta|}}}\widetilde{\omega}(\zeta)\mathrm{d}A(\zeta)\right\|_{L_\omega^1}
\lesssim \|g\|_{A_\omega^1} .$$
It follows from \eqref{E2} that for $g\in A_\omega^1$, it holds
$$
\|T_f^\omega (g)\|_{A_\omega^1}\lesssim
\left\|\int_{\mathbb{D}}\frac{G_z^\omega(\zeta)}{\log{\frac{e}{1-|\zeta|}}}\widetilde{\omega}(\zeta)\mathrm{d}A(\zeta)\right\|_{L_\omega^1}
\lesssim\|g\|_{A_\omega^1},
$$
which completes the proof of Theorem \ref{main2}(i).

In a way similar to the proof of Theorem \ref{main1}, we derive that $\|T_f^\omega g_n\|_{A_\omega^1} \rightarrow 0$ as $n \rightarrow \infty$ if $\{g_n\}_{n=1}^\infty$ is any bounded sequence in $A_\omega^1$ that converges to $0$ uniformly on any compact subset of $\mathbb{D}$. By a standard density argument, this implies the compactness of $T_f^{\omega}$ on $A_{\omega}^1$.
\qed

\subsection{The boundedness of Toeplitz operator with symbols in $\mathrm{BMO}^p_{\omega,\log}\cap L^\infty$}
In this section, we first provide a sufficient condition for the boundedness of Toeplitz operator $T_f^\omega:A_\omega^1 \rightarrow A_\omega^1$ with $\mathrm{BMO}^p_{\omega,\log}\cap L^\infty$ symbol for $\omega\in \mathcal{D}$. Furthermore, when $f$ is also analytic on $\mathbb{D}$, i.e. $f\in \mathcal{LB}\cap H^\infty$, we also provide the equivalent characterization of the boundedness of $T_{\bar{f}}$ acting on $A_\omega^1$. First we define some function spaces as follows.

Let $\omega$ be a radial weight and $0<r<\infty$ such that $\omega(D(z,r))>0$. We define the logarithmic mean oscillation space, denoted by
$\mathrm{BMO}_{\omega, r,\log }^p$, to be the space consisting of all functions $f$ in $\mathrm{BMO}_{\omega,r}^p$ for which
$$
\|f\|_{\mathrm{BMO}^p_{\omega,r,\log }}=\sup _{z \in \mathbb{D}} \bigg(\Big(\log{\frac{e}{1-|z|}}\Big)\Big(\mathrm{M O}_{\omega,r}^p(f)(z)\Big)\bigg)<\infty.
$$
The logarithmic $\mathrm{BO}$ space, denoted by $\mathrm{BO}_{r, \log }$, consists of all $f \in\mathrm{BO}$ for which
$$
\|f\|_{\mathrm{BO}_{r, \log }}=\sup _{z \in \mathbb{D}}\bigg( \Big(\log{\frac{e}{1-|z|}}\Big)\Omega_r f(z)\bigg)<\infty.
$$
For $0<p<\infty$ and $0<r<\infty$, we say that $f \in \mathrm{BA}_{\omega,r,\log}^p$ if $f\in L^p_\omega$ and
$$
\|f\|_{\mathrm{BA}_{\omega, r, \log }^p}=\sup _{z \in \mathbb{D}} \left(\Big(\log{\frac{e}{1-|z|}}\Big)\left(\frac{1}{\omega(D(z,r))} \int_{D(z, r)}|f(\zeta)|^p
\omega(\zeta)\mathrm{d} A(\zeta)\right)^{1 / p}\right)<\infty.
$$
For $\omega\in\mathcal{D}$, there exists an $r_0=r_0(\omega)$ such that \eqref{Dweight2} is satisfied.
Using the argument in \cite{PPR} with minor modifications, for such $r_0$, the above spaces can be shown to be independent of $r$ with $r\geq r_0$, so
we use $
\mathrm{BMO}_{\omega, \log }^p$, $\mathrm{BO}_{\log }$ and $\mathrm{BA}_{\omega,\log }^p$ instead.
For $\omega\in \mathcal{D}$, the relationship between $\mathrm{BA}_{\omega,\log}^1$ and the Toeplitz operator on $A_{\omega}^1$ was investigated in \cite{DGWW2}.
\begin{lemma}[Propositon 4.1 in \cite{DGWW2}]\label{BAlog}
Let $\omega$ be a $\mathcal{D}$ weight and $f\in L_\omega^1$. If there exists $r\in(0,1)$ such that $$
\sup_{z\in\mathbb{D}} \frac{\log\frac{e}{1-|z|}}{\omega(D(z,r))} \int_{D(z,r)} |f(\zeta)| \omega(\zeta)\mathrm{d}A(\zeta)<\infty,
$$
then the Toeplitz operator $T_f^\omega$ is bounded on $A_\omega^1$.
\end{lemma}

Analogously to the case of $\mathrm{BMO}_{\omega}^p$, we give the following decomposition of $\mathrm{BMO}^p_{\omega,\log}$ with $\mathcal{D}$ weights. The proof is similar to \cite[Theorem 11]{PPR} and thus is omitted.

\begin{proposition} Let  $\omega\in\mathcal{D}$ and $1 \leqslant p<\infty$. Then there exists $r_0=r_0(\omega)$ such that $\forall r\geq r_0$, it holds $\mathrm{BMO}_{\omega,
\log }^p=\mathrm{BO}_{\log }+\mathrm{BA}_{\omega,\log}^p$. Moreover, if $f=f_1+f_2$ with $f_1 \in \mathrm{BO}_{\log }$ and $f_2 \in \mathrm{BA}_{\omega,\log
}^p$, then $f\in\mathrm{BMO}_{\omega,
\log }^p$ with $\|f\|_{\mathrm{BMO}_{\omega,\log }^p} \asymp$ $\left\|f_1\right\|_{\mathrm{BO}_{\log }}+\left\|f_2\right\|_{\mathrm{BA}_{\omega,\log }^p}$.
\end{proposition}
Before we proceed, we give the following useful result.
\begin{lemma}\label{logBMO}
Let $\omega\in \mathcal{D} $, $1<p<\infty$ and  $r_0$ be chosen as in \eqref{Dweight2}. If $f\in L^{\infty} \cap \mathrm{BMO}^p_{\omega,r,\log}$, then $fg\in\mathrm{BMO}^p_{\omega,r}$ for each $g \in
\mathrm{BMO}^p_{\omega,r}$ and $r\geq r_0$.

\end{lemma}
\begin{proof}
By the triangle inequality, for $\zeta,z\in\mathbb{D}$ we have
\begin{eqnarray}\label{G1}
\left|f(\zeta) g(\zeta)-(\widehat{f g})_{r,\omega}(z)\right|
&\leq & \|f\|_{\infty}|g(\zeta)-\widehat{g}_{r,\omega}(z)|
+|\widehat{g}_{r,\omega}(z)||f(\zeta)-\widehat{f}_{r,\omega}(z)|\nonumber\\
&\quad&+\,|\widehat{g}_{r,\omega}(z) \widehat{f}_{r,\omega}(z)-(\widehat{f g})_{r,\omega}(z)|.
\end{eqnarray}
By H\"{o}lder's inequality, for $z\in\mathbb{D}$, we obtain
\begin{eqnarray}\label{G2}
\left|\widehat{g}_{r,\omega}(z) \widehat{f}_{r,\omega}(z)-(\widehat{fg})_{r,\omega}(z)\right|
& \leq&\|f\|_{\infty}\left(\frac{1}{\omega(D(z,r))} \int_{D(z,r)}\left|g(u)-\widehat{g}_{r,\omega}(z)\right|^p \omega(u)\mathrm{d} A(u)\right)^{1 / p}\nonumber\\
&\leq&\|f\|_{\infty}\|g\|_{\mathrm{BMO}^p_{\omega,r}}.
\end{eqnarray}
For $r\geq r_0$, it is obvious that $f\in\mathrm{BMO}^p_{\omega,\log}\subset \mathrm{BMO}^p_{\omega}$. Then combining \eqref{G1} and \eqref{G2} gives
\begin{equation}\label{G3}
\mathrm{MO}^p_{\omega,r}(f g)(z)
\lesssim
\|f\|_{\infty}\|g\|_{\mathrm{BMO}^p_{\omega}}+\left|\widehat{g}_{r,\omega}(z)\right| \mathrm{MO}^p_{\omega,r}(f)(z),\quad{z\in\mathbb{D}}.
\end{equation}
Since $\omega\in \mathcal{D}$, by the decomposition of $\mathrm{BMO}^p_{\omega}=\mathrm{BA}^p_{\omega}+\mathrm{BO}$, we have $\widehat{g}_{r,\omega}$
belongs to $\mathrm{BO}$ (see \cite{PPR}). Then \eqref{BO} gives
\begin{equation}\label{G4}
	|\widehat{g}_{r,\omega}(z)-\widehat{g}_{r,\omega}(0)|
	\lesssim \Big(1+\beta(0, z)\Big)\|\widehat{g}_{r,\omega}\|_{\mathrm{BO}}
	\lesssim \Big(1+\beta(0, z)\Big)\|g\|_{\mathrm{BMO}^p_{\omega}},\quad z \in \mathbb{D}.
\end{equation}
It follows from \eqref{G3} and \eqref{G4} that
\begin{align}
\mathrm{MO}^p_{\omega,r}(f g)(z)
&\lesssim \|f\|_{\infty}\|g\|_{\mathrm{BMO}^p_{\omega}}+\|f\|_{\mathrm{BMO}^p_{\omega,\log}}\|g\|_{\mathrm{BMO}^p_{\omega}}+|\widehat{g}_{r,\omega}(0)|\mathrm{MO}^p_{\omega,r}(f)(z)\label{eq3.31}\\
& \lesssim \|f\|_{\infty}\|g\|_{\mathrm{BMO}^p_{\omega}}+\|f\|_{\mathrm{BMO}^p_{\omega,\log}}\|g\|_{\mathrm{BMO}^p_{\omega}}+ \frac{\|g\|_{L_\omega^p}}{\omega(\mathbb{D})}\|f\|_{\mathrm{BMO}^p_{\omega}},\nonumber
\end{align}
which completes the proof of lemma.
\end{proof}

Recall that the Bloch space $\mathcal{B}$ consists of $f\in H(\mathbb{D})$ such that
$$
\|f\|_{\mathcal{B}}=\sup_{z\in\mathbb{D}}|f^{\prime}(z)|(1-|z|^2)+|f(0)|<\infty.
$$
The dual space of $A_\omega^1$ can be identified with the Bloch space under the pairing of $A_\omega^2$ for $\mathcal{D}$ weights.
\begin{lemma}[\cite{PR2}]\label{Bloch3}
Let $\omega\in\mathcal{D}$. Then $(A_\omega^1)^*\simeq \mathcal{B}$ via the $A_\omega^2$-pairing
$$
\langle f,g \rangle_{L_\omega^2}=\lim_{r \rightarrow 1^-}\int_{\mathbb{D}}f_r(\zeta)g(\zeta)\omega(\zeta)\mathrm{d}A(\zeta), \quad f\in A_\omega^1,\,\,g\in \mathcal B,
$$
where $f_r(\zeta)=f(r\zeta), {r\in(0,1)}.$
\end{lemma}
The following result is well-known; see \cite{Zhu2} for example.
\begin{lemma}\label{Bloch}
An analytic function $f$ on $\mathbb{D}$ belongs to $\mathcal{B}$ if and only if there exists
a positive constant $C$ such that $|f(z)-f(w)| \leq C \beta(z, w)$ for all $z$ and $w$
in $\mathbb{D}$.
\end{lemma}
Recall that for $1<p<\infty$ and $\omega\in \mathcal{D}$, there exists $r_0=r_0(\omega)$ such that $\mathrm{BMO}_{\omega}^p=\mathrm{BMO}_{\omega,r}^p$ for each $r\geq r_0$.
The following result characterized the boundedness of Bergman projection $P_\omega$ from $\mathrm{BMO}_{\omega}^p$ to $\mathcal{B}$.

\begin{lemma}[\cite{PR1}]\label{Bloch2}
Let $1<p<\infty$ and $\omega\in \mathcal{D}$. Then $P_\omega: \mathrm{BMO}_{\omega}^p\rightarrow \mathcal{B}$ is bounded.
\end{lemma}
For $\omega\in \mathcal{D} $, $1<p<\infty$ and  $r_0$ is chosen as in \eqref{Dweight2}, by Lemma \ref{Bloch2}, if $f\in  \mathrm{BMO}_{\omega}^p$, then $\|P_\omega(f)\|_{\mathcal{B}}\lesssim \|f\|_{\mathrm{BMO}_{\omega}^p}$. For $f\in A_\omega^p$, it follows that
$\|f\|_{\mathcal{B}}\lesssim \|f\|_{\mathrm{BMO}_{\omega}^p}$, which implies that
$$\mathrm{BMO}_{\omega}^p\cap H(\mathbb{D})\subset \mathcal{B}.$$
Conversely, it is obvious that $\mathcal{B}\subset\mathrm{BO}\subset \mathrm{BMO}^p_{\omega}$. We conclude that
$$
\mathrm{BMO}_{\omega}^p\cap H(\mathbb{D})=\mathcal{B}.
$$ Moreover, for $f\in H(\mathbb{D})$, it holds that
\begin{equation}\label{G8}
\|f\|_{\mathrm{BMO}_{\omega}^p}\asymp\|f\|_{\mathcal{B}}.
\end{equation}

Now we are ready to deal with the Toeplitz operator with symbol in $L^{\infty} \cap \mathrm{BMO}^p_{\omega,\log}$.

\begin{proposition}\label{main3}
Let $\omega\in\mathcal{D}$ and $1<p<\infty$. If $f \in L^{\infty} \cap \mathrm{BMO}^p_{\omega,\log}$, then $T_{f}^\omega$ is bounded on $A_\omega^1$. Moreover,
$$
\left\|T_f^\omega\right\|_{A_\omega^1\rightarrow A_\omega^1} \lesssim \|f\|_{\infty}+\|f\|_{\mathrm{BMO}^p_{\omega,\log }}.
$$
\end{proposition}
\begin{proof}
Let  $\omega\in\mathcal{D}$, $g \in \mathcal{B}$ and $r_0$ be chosen as in \eqref{Dweight2}. For $1<p<\infty$ and $r\geq r_0$, \eqref{G8} together with \eqref{eq3.31} and Lemma \ref{Bloch2} implies

\begin{eqnarray}\label{G5}
\left\|T_f^\omega g\right\|_{\mathcal{B}} & =&\|P_\omega(f g)\|_{\mathcal{B}} \lesssim \|f g\|_{\mathrm{BMO}^p_{\omega,r}} \nonumber\\
& \lesssim&\|f\|_{\infty}\|g\|_{\mathrm{BMO}^p_{\omega,r}}+
\|f\|_{\mathrm{BMO}^p_{\omega,\log}}\|g\|_{\mathrm{BMO}^p_{\omega,r}}+|\widehat{g}_{r,\omega}(0)|\sup _{z \in \mathbb{D}}\mathrm{MO}^p_{\omega,r}(f)(z)\nonumber\\
&\asymp& \|f\|_{\infty}\|g\|_{\mathcal{B}}+
\|f\|_{\mathrm{BMO}^p_{\omega,\log}}\|g\|_{\mathcal{B}}+|\widehat{g}_{r,\omega}(0)|\sup _{z \in \mathbb{D}}\mathrm{MO}^p_{\omega,r}(f)(z).
\end{eqnarray}
By \eqref{Dweight2} and \eqref{Littlewood}, we have
\begin{equation}\label{G6}
|\widehat{g}_{r,\omega}(0)|\lesssim \frac{\|g\|_{A_\omega^1}}{\omega(\mathbb{D})}\asymp
\int_{\mathbb{D}}|g^{\prime}(z)|(1-|z|)\omega(z)\mathrm{d}A(z)+|g(0)|\lesssim \|g\|_{\mathcal{B}}.
\end{equation}
Then by \eqref{G5} and \eqref{G6}, we get
$$
\|T_f^\omega\|_{\mathcal{B}\rightarrow\mathcal{B}} \lesssim \|f\|_{\infty}+\|f\|_{\mathrm{BMO}^p_{\omega,\log }}.
$$
Since $\omega\in\mathcal{D}$, it follows from the duality $(A_\omega^1)^* \simeq \mathcal{B}$ via the $A_\omega^2$-paring in Lemma \ref{Bloch3} that $T^\omega_f: A_\omega^1 \rightarrow A_\omega^1$ is bounded and
$$
\|T_f^\omega\|_{A_\omega^1\rightarrow A_\omega^1} \lesssim \|f\|_{\infty}+\|f\|_{\mathrm{BMO}^p_{\omega,\log }}.
$$
This completes the proof.
\end{proof}

Since $\mathrm{BO}_{\log}\subseteq \mathrm{BMO}_{\omega, \log }^p$, if $f\in \mathrm{BO}_{\log}\cap L^\infty$, then $T_f^\omega:
A_\omega^1\rightarrow A_\omega^1$ is bounded.
Notice that $\mathrm{BA}^p_{\omega,\log}\subset \mathrm{BA}^1_{\omega,\log}$ for $1<p<\infty$. By Lemma \ref{BAlog}, we get the following corollary.
\begin{corollary}
Let $\omega\in\mathcal{D}$ and $1<p<\infty$. If $f=f_1+f_2$, where $f_1 \in \mathrm{BO}_{\log } \cap L^{\infty}$ and $f_2\in \mathrm{BA}^p_{\omega,\log}$, then $T_f^\omega: A_\omega^1
\rightarrow A_\omega^1$ is bounded.
\end{corollary}

To obtain an equivalent conditon for the boundedness of Toeplitz operator with co-analytic symbol on $A_\omega^1$, we need the following results.

For $\omega, \nu$ in $\mathcal{D}$, it is shown in \cite{Perala} that there exists a unique linear operator $R^{\omega,\nu}$ on $H(\mathbb{D})$ equipped with the compact-open topology, called the general fractional derivative, which satisfies:
\begin{itemize}
\item [(i)] $R^{\omega, \nu}: H(\mathbb{D}) \rightarrow H(\mathbb{D}) $ is continuous;
\item [(ii)] $R^{\omega, \nu} B_z^\omega(\xi)=B_z^\nu(\xi)$;
\item [(iii)] For $f=\sum f_k z^k\in H(\mathbb{D})$, it holds that
\end{itemize}
\begin{equation}\label{Rwv}
R^{\omega, \nu} f(z)=\int_{\mathbb{D}} f(\zeta) \overline{B_z^\nu(\zeta)} \omega(\zeta) \mathrm{d}
A(\zeta)=\sum_{k=0}^{\infty}\left(\frac{\omega_{2k+1}}{\nu_{2k+1}}\right) f_k z^k.
\end{equation}
For  $\omega, \nu$ in $\mathcal{D}$, the general fractional derivative $R^{\omega,\nu}$ is a bijection on $H(\mathbb{D})$, and satisfies $(R^{\omega, \nu})^{-1}=R^{\nu,\omega}$. Moreover, it follows directly from \cite[Lemma 3]{PR3} that for
$\omega\in\mathcal{D}$ and $\beta>0$, we have $\omega_{[\beta]} \in \mathcal{D}$. In particular, if $\nu=\omega_{[\beta]}$, we briefly denote $R^{\omega,\nu}$ by $R^{\omega,\beta}$.

We denote by $C_c(\mathbb{D})$ the space of continuous functions with compact support in $\mathbb{D}$. To characterize the boundedness of $T_f^\omega$ on $A_\omega^1$, we need the following lemma which is inspired by \cite{ABT}.
\begin{lemma}
Let $\omega\in \mathcal{D}$ and $\beta>0$. Then $P_{\omega_{[\beta]}}(C_c(\mathbb{D}))$ is dense in $A_\omega^1$.
\end{lemma}
\begin{proof}
It is obvious that $P_{\omega_{[\beta]}}(C_c(\mathbb{D}))\subset H^\infty \subset A_\omega^1$. Since $\omega\in \mathcal{D}$, by Lemma \ref{Bloch3}, we only need to show that if $h\in\mathcal{B}$ satisfies
$$
\int_{\mathbb{D}}P_{\omega_{[\beta]}}\phi(\zeta)\overline{h(\zeta)}\omega(\zeta)\mathrm{d}A(\zeta)=0, \quad{\forall\phi\in C_c(\mathbb{D})},
$$
then $h\equiv0$. In fact, by Fubini's theorem and the representation of $R^{\omega,\beta}$ in \eqref{Rwv}, we have
$$
\begin{aligned}
0=&\int_{\mathbb{D}}P_{\omega_{[\beta]}}\phi(\zeta)\overline{h(\zeta)}\omega(\zeta)\mathrm{d}A(\zeta)\\
=&\int_{\mathbb{D}}\int_{\mathbb{D}}
\phi(u)\overline{B_\zeta^{\omega_{[\beta]}}(u)}\omega_{[\beta]}(u)\mathrm{d}A(u)\overline{h(\zeta)}\omega(\zeta)\mathrm{d}A(\zeta)\\
=&\int_{\mathbb{D}}\phi(u)\overline{R^{\omega,\beta}h(u)}(1-|u|)^\beta\omega(u)\mathrm{d}A(u),\quad{\forall\phi\in C_c(\mathbb{D})}.
\end{aligned}
$$
It follows that $R^{\omega,\beta}h= 0$ almost everywhere on $\mathbb{D}$. Since $R^{\omega,\beta}h\in H(\mathbb{D})$, then $R^{\omega,\beta}h\equiv 0$ on $\mathbb{D}$. By the bijectivity of $R^{\omega,\beta}$ on $H(\mathbb{D})$, we have $h\equiv 0$ on $\mathbb{D}$. This completes the proof.
\end{proof}
Recall that for a finite Borel measure $\mu$, the Toeplitz operator $T^\omega_{\mu}$ induced by $\mu$ is defined as
$$
T^\omega_{\mu}(f)(z)=\int_{\mathbb{D}} f(\zeta)\overline{B_z^\omega(\zeta)}\mathrm{d}\mu(\zeta).
$$
The following lemma is a useful tool for the characterization of the boundedness of an operator acting on $A_\omega^1$, which is a generalization of \cite[Theorem 1.4]{ABT}. The proof is a minor modification of the original one and is omitted.

\begin{lemma}\label{Ab}
Let $\omega\in\mathcal{D}$, $\beta>0$ and $A$ be a linear operator from $H^\infty$ to $H(\mathbb{D})$. For the following two assertions:
\begin{itemize}
\item[(i)] $A$ extends to a bounded operator on $A_\omega^1$,
\item[(ii)] $\sup_{\zeta\in\mathbb{D}}\big\|A \widetilde{B}_{\zeta}^{\omega_{[\beta]}}\big\|_{A_\omega^1}<\infty,$
\end{itemize}
it holds that (i) implies (ii). Moreover, (ii) implies (i) in the following two cases.
\begin{itemize}
\item[(a)] The identity
$$
\int_{\mathbb{D}} A \widetilde{B}_{\zeta}^{\omega_{[\beta]}}(z)g(\zeta)\omega(\zeta)\mathrm{d}A(\zeta)=Ag(z)
$$
holds for all $z\in\mathbb{D}$ and $g\in P_{\omega_{[\beta]}}(C_c(\mathbb{D}))$.

\item[(b)] $A=T^\omega_\mu$, where $\mu$ is a complex Borel measure on $\mathbb{D}$.
\end{itemize}
\end{lemma}
Now we are in a position to prove Theorem \ref{TLB}.

\vspace{5pt}
\noindent \textbf{Proof of Theorem \ref{TLB}.}
We first prove that (i) implies (ii). Suppose $T_{\bar{f}}^\omega$ is bounded on $A_\omega^1$. Then there exists a positive constant $C=C(\omega)$ such that
$$
|\langle T_{\bar{f}}^\omega(g),h\rangle_{L_\omega^2}|\leq C\|g\|_{A_\omega^1}\|h\|_{\mathcal{B}},\quad{g,h\in H^\infty}.
$$
By Fubini's theorem, we obtain
\begin{equation}\label{TL1}
\left|\int_{\mathbb{D}}\overline{f(\zeta)}g(\zeta)\overline{h(\zeta)}\omega(\zeta)\mathrm{d}A(\zeta)\right|\leq C\|g\|_{A_\omega^1}\|h\|_{\mathcal{B}},\quad{g,h\in
H^\infty}.
\end{equation}
For each $z\in\mathbb{D}$, let $g_{1,z}(\zeta)=\zeta (B_z^\omega)^{\prime}(\zeta)$ and $h_z(\zeta)=\log\frac{e}{1-\bar{z}\zeta}$ for $\zeta\in\mathbb{D}$.
Then $\|h_z\|_{\mathcal{B}}\lesssim 1$ uniformly in $z$. By Lemma \ref{Ne1}, we have
$$
\|g_{1,z}\|_{A_\omega^1}\leq\|(B_z^\omega)^{\prime}\|_{A_\omega^1}\asymp \frac{1}{1-|z|},\quad{z\in\mathbb{D}}.
$$
Plugging $g_{1,z}$ and $h_z$ into \eqref{TL1} gives
\begin{equation}\label{TL}
\left|\int_{\mathbb{D}}f(\zeta)\overline{\zeta (B_z^\omega)^{\prime}(\zeta)}\Big(\log\frac{e}{1-\bar{z}\zeta}\Big)\omega(\zeta)\mathrm{d}A(\zeta)\right|
\lesssim \frac{1}{1-|z|},\quad{z\in\mathbb{D}}.
\end{equation}
Furthermore, Lemma \ref{Ne1} shows
\begin{equation}\label{TL5}
\|B_z^\omega\|_{A_\omega^1}\asymp \log\frac{e}{1-|z|},\quad z\in \mathbb{D}.
\end{equation}
Let $g_{2,z}(\zeta)=B_z^\omega(\zeta)$ for $\zeta\in\mathbb{D}$. Plugging $g_{2,z}$ and $h_z$ into \eqref{TL1}, it follows from the reproducing formula and \eqref{TL5} that $f\in H^\infty$.
Notice that
\begin{equation}\label{TL2}
	\zeta (B_z^\omega)^{\prime}(\zeta)=\overline{z(B_{\zeta}^\omega)^{\prime}(z)},\quad{z,\zeta \in \mathbb{D}}.
\end{equation}
Since the function $F_z(\zeta)=f(\zeta)\log\frac{e}{1-\bar{z}\zeta}$  is analytic for $\zeta\in\mathbb{D}$ and
$$
\|F_z\|_{A_\omega^1}=\int_{\mathbb{D}}|f(\zeta)|\Big(\log\frac{e}{|1-\bar{z}\zeta|}\Big)\omega(\zeta)\mathrm{d}A(\zeta)\leq \Big(\log\frac{e}{1-|z|}\Big)\|f\|_{A_\omega^1}<\infty,
$$
the reproducing formula together with \eqref{TL} and \eqref{TL2} shows that
$$
\begin{aligned}
\frac{1}{1-|z|}\gtrsim
&\left|z\int_{\mathbb{D}}F_z(\zeta)(B_{\zeta}^\omega)^{\prime}(z)\omega(\zeta)\mathrm{d}A(\zeta)\right|
=\left|zF_z^{\prime}(z)\right|\\
=&\left|z\bigg( f^{\prime}(z)\log\frac{e}{1-|z|}+f(z)\frac{\bar{z}}{1-|z|}\bigg) \right|.
\end{aligned}
$$
It follows that
$$
\left|(1-|z|)f^{\prime}(z)\log\frac{e}{1-|z|}+\bar{z}f(z)\right|\lesssim \frac{1}{|z|}.$$
This combined with the triangle inequality and the fact that $f\in H^\infty$ gives
$$
\begin{aligned}
(1-|z|)|f^{\prime}(z)|\log\frac{e}{1-|z|}
&\leq \Big|(1-|z|)f^{\prime}(z)\log\frac{e}{1-|z|}+\bar{z}f(z)\Big|+|\bar{z}f(z)|\\
&\lesssim \frac{1}{|z|}+|f(z)| \lesssim 1+\|f\|_{\infty}\lesssim 1, \quad{|z|\rightarrow 1^-}.
\end{aligned}
$$
Thus, $f\in \mathcal{LB}$, as desired.

Next, we show that (ii) implies (i).
Note that $\mathcal{LB}\subset\mathcal{B}$. For every $g\in \mathcal{B}$, the triangle inequality and Lemma \ref{Bloch} yield
$$
|(fg)^{\prime}(z)|\leq\|f\|_{\infty}|g^{\prime}(z)|+|g(z)||f^{\prime}(z)|
\lesssim \|f\|_{\infty}|g^{\prime}(z)|+\beta(0,z)|f^{\prime}(z)|\|g\|_{\mathcal{B}}+|g(0)||f^{\prime}(z)|,\quad{z\in\mathbb{D}}.
$$
It follows that
\begin{equation}\label{TL6}
(1-|z|)|(fg)^{\prime}(z)|\lesssim
\|f\|_{\infty}\|g\|_{\mathcal{B}}+(1-|z|)\log\frac{e}{1-|z|}|f^{\prime}(z)|\|g\|_{\mathcal{B}}+|g(0)|\|f\|_{\mathcal{B}},\quad{z\in\mathbb{D}},
\end{equation}
which implies that $f\mathcal{B}\subseteq\mathcal{B}$. Moreover, \eqref{TL6} implies that for $g\in \mathcal{B}$ we have
$$
\|T_f^\omega(g)\|_{\mathcal{B}}=\|P_\omega(fg)\|_{\mathcal{B}}=\|fg\|_{\mathcal{B}}\lesssim \|g\|_{\mathcal{B}}.
$$
Combining this with the duality relationship between $A_\omega^1$ and $\mathcal{B}$ as in Lemma \ref{Bloch3}, we conclude that $T_{\bar{f}}^\omega: A_\omega^1 \rightarrow A_\omega^1$ is bounded, as desired.

The equivalence between (i) and (iii) is a consequence of Lemma \ref{Ab}. The proof is completed.
\qed

\section{The characterizations of bounded Hankel operators}
In this section we turn to investigate the boundedness of the Hankel operator $H_{\bar{f}}^\omega:A_\omega^1\rightarrow L_\omega^1$ with co-analytic symbols.

\vspace{5pt}
\noindent \textbf{Proof of Theorem \ref{HLB}.}
Necessity. Suppose $f\in A_\omega^2$ and $H_{\bar{f}}^\omega$ is bounded from $A_\omega^1$ to $L_\omega^1$. Then there exists a positive constant $C=C(\omega)$ such
that
$$
\|H_{\bar{f}}^\omega(g)\|_{L_\omega^1}\leq C\|g\|_{A_\omega^1}, \quad{g\in H^\infty}.
$$
By the duality of $(L_\omega^1)^*$ with $L^\infty$ under the $L_\omega^2$-pairing, we obtain
\begin{equation}\label{HL4}
\left|\int_{\mathbb{D}}H_{\bar{f}}^\omega(g)(\zeta)\overline{h(\zeta)}\omega(\zeta)\mathrm{d}A(\zeta)\right|
\leq\|H_{\bar{f}}^\omega(g)\|_{L_\omega^1}\|h\|_{\infty}
\leq C\|g\|_{A_\omega^1}\|h\|_{\infty},\quad g\in H^\infty, h\in L^\infty.
\end{equation}
Let $H_0^\infty$ denote the space of all functions in $H^{\infty}$ that vanish at the origin. For $h_1\in H_0^\infty$, it holds that $P_\omega(\overline{h_1})=0$. It follows that for $g\in H^\infty$ we have
\begin{eqnarray}\label{HL5}
\int_{\mathbb{D}}H_{\bar{f}}^\omega(g)(\zeta)h_1(\zeta)\omega(\zeta)\mathrm{d}A(\zeta)
&=&\int_{\mathbb{D}}\overline{f(\zeta)}g(\zeta)h_1(\zeta)\omega(\zeta)\mathrm{d}A(\zeta)-\int_{\mathbb{D}} P_\omega(\bar{f}g)(\zeta)h_1(\zeta)\omega(\zeta)\mathrm{d}A(\zeta)\nonumber\\
&=&\int_{\mathbb{D}}\overline{f(\zeta)}g(\zeta)h_1(\zeta)\omega(\zeta)\mathrm{d}A(\zeta)-\int_{\mathbb{D}}\overline{f(\zeta)}g(\zeta) \overline{P_\omega(\widebar{h_1})(\zeta)} \omega(\zeta)\mathrm{d}A(\zeta)\nonumber\\
&=&\int_{\mathbb{D}}\overline{f(\zeta)}g(\zeta)h_1(\zeta)\omega(\zeta)\mathrm{d}A(\zeta), \quad{h_1\in H_0^\infty}.
\end{eqnarray}
Similarly, we also have
\begin{equation}\label{HL6}
\int_{\mathbb{D}}H_{\bar{f}}^\omega(g)(\zeta)\overline{h_2(\zeta)}\omega(\zeta)\mathrm{d}A(\zeta)=0,\quad{g,h_2\in H^\infty}.
\end{equation}
For $g\in H^\infty$, let $g_1(\zeta)=\zeta g(\zeta)$ for $\zeta \in\mathbb{D}$. Plugging the function $h_1(\zeta)=\zeta$ into \eqref{HL5}, we conclude from \eqref{HL4} that
$$
\left|\int_{\mathbb{D}}\overline{f(\zeta)}g_1(\zeta)\omega(\zeta)\mathrm{d}A(\zeta)\right|\leq C\|g\|_{A_\omega^1}\lesssim \|g_1\|_{A_\omega^1}, \quad{g_1\in H_0^\infty}.
$$
The above inequality also holds if $g$ is the constant function. It follows that the densely defined linear functional
$$F_f(g)=\int_{\mathbb{D}}\overline{f(\zeta)}g(\zeta)\omega(\zeta)\mathrm{d}A(\zeta),\quad{g\in H^\infty}$$
extends to a bounded linear functional on $A_\omega^1$. It follows from Lemma \ref{Bloch3} that $f\in\mathcal{B}$.
Fix $z\in \mathbb{D}$. Letting $h_2$ in \eqref{HL5} be the function $h_{2,z}(\zeta)=\log (1-\bar{z} \zeta)$, we deduce from \eqref{HL4} and \eqref{HL6} that
\begin{eqnarray}\label{HL}
\left|\int H_{\bar{f}}^\omega(g)(\zeta)\log(1-\bar{z}\zeta)  \omega(\zeta) \mathrm{d} A(\zeta)\right|
&=&\left|\int H_{\bar{f}}^\omega(g)(\zeta)\big(\log(1-\bar{z}\zeta)- \overline{\log(1-\bar{z}\zeta)}\big) \omega(\zeta) \mathrm{d} A(\zeta)\right|\nonumber\\
&\leq& 2C\|g\|_{A_\omega^1} \sup_{\zeta \in\mathbb{D}}|\mathrm{Im}\log (1-\bar{z} \zeta)|
\lesssim \|g\|_{A_\omega^1},\,\,\,{g \in H^{\infty}}.
\end{eqnarray}
\eqref{HL5} combined with \eqref{HL} shows that
$$
\left|\int \overline{f(\zeta)} g(\zeta) \log (1-\bar{z} \zeta)\omega(\zeta) \mathrm{d} A(\zeta)\right| \leq C\|g\|_{A_\omega^1}, \quad g \in H^{\infty}.
$$
Choosing $g_{2,z}(\zeta)=\zeta (B_z^\omega)^{\prime}(\zeta)$, we conclude from Lemma \ref{Ne1} and the above inequality that
\begin{equation}\label{HL7}
\left|\int f(\zeta) \overline{\zeta (B_z^\omega)^{\prime}(\zeta)}  \Big(\log \frac{1}{(1-\bar{\zeta}z)}\Big)\omega(\zeta) \mathrm{d} A(\zeta)\right|\lesssim
\frac{1}{1-|z|},\quad{z\in\mathbb{D}}.
\end{equation}
Next we show that
$$
\left|\int_{\mathbb{D}}f(\zeta)\overline{\zeta (B_z^\omega)^{\prime}(\zeta)}\left(\log\frac{1-\bar{\zeta}z}{1-|z|}\right)\omega(\zeta)\mathrm{d}A(\zeta)\right|\lesssim
\frac{1}{1-|z|},\quad{z\in\mathbb{D}}.
$$
Let $g_{3,z}(\zeta)=\zeta (B_z^\omega)^{\prime}(\zeta)\log\frac{1-\bar{z}\zeta}{1-|z|}$ and $0<\varepsilon < \frac{2}{\beta+3}$, where $\beta$
is chosen as in Lemma \ref{Dhat}. Since $f\in\mathcal{B}$ and $\omega\in\mathcal{D}$, then by Lemma \ref{Bloch3},
$$
\left|\int_{\mathbb{D}}f(\zeta)\overline{\zeta (B_z^\omega)^{\prime}(\zeta)}\left(\log\frac{1-\bar{\zeta}z}{1-|z|}\right)\omega(\zeta)\mathrm{d}A(\zeta)\right|
=|\langle f,g_{3,z}\rangle_{L_\omega^2}|
\lesssim\|f\|_{\mathcal{B}}\|g_{3,z}\|_{A_\omega^1}.
$$
Notice that for $z,\zeta\in \mathbb{D}$, $\log\Big(\frac{|1-\bar{z}\zeta|}{1-|z|}\Big)^\varepsilon\leq \Big(\frac{|1-\bar{\zeta}z|}{1-|z|}\Big)^\varepsilon$. By H\"{o}lder's inequality and \eqref{TL2},  we deduce
\begin{eqnarray}\label{HL8}
\left|\int_{\mathbb{D}}f(\zeta)\overline{\zeta (B_z^\omega)^{\prime}(\zeta)}\left(\log\frac{1-\bar{\zeta}z}{1-|z|}\right)\omega(\zeta)\mathrm{d}A(\zeta)\right|
&\lesssim&\frac{\|f\|_{\mathcal{B}}}{(1-|z|)^\varepsilon}\int_{\mathbb{D}}|(1-\bar{\zeta}z)|^{\varepsilon}|(B_{\zeta}^\omega)^{\prime}(z)|
\omega(\zeta)\mathrm{d}A(\zeta) \nonumber\\
&\leq&\frac{\|f\|_{\mathcal{B}}}{(1-|z|)^\varepsilon}I_1(z)^{\frac{\varepsilon}{2}}I_2(z)^{\frac{2-\varepsilon}{2}},\quad{z\in\mathbb{D}},
\end{eqnarray}
where
$$
I_1(z)=\int_{\mathbb{D}}|(1-\bar{\zeta}z)(B_{\zeta}^\omega)^{\prime}(z)|^2\omega(\zeta)\mathrm{d}A(\zeta),\quad{z\in\mathbb{D}},
$$
and
$$
I_2(z)=\int_{\mathbb{D}}|(B_{\zeta}^\omega)^{\prime}(z)|^{\frac{2(1-\varepsilon)}{2-\varepsilon}}\omega(\zeta)\mathrm{d}A(\zeta),\quad{z\in\mathbb{D}}.
$$
First we estimate $I_1(z)$.
From \cite{PR1}, we have
\begin{eqnarray}\label{1-zzeta}
	2(1-\bar{\zeta} z)(B_\zeta^\omega)^{\prime}(z)
	& =&\bar{\zeta}\left(\sum_{k=1}^{\infty} \frac{k(\bar{\zeta} z)^{k-1}}{\omega_{2
			k+1}}-\sum_{k=1}^{\infty} \frac{k(\bar{\zeta} z)^k}{\omega_{2 k+1}}\right) \nonumber\\
& =&\bar{\zeta}\left(J_1+J_2(z, \bar{\zeta})+J_3(z, \bar{\zeta})\right), \quad z, \zeta \in \mathbb{D},
\end{eqnarray}
where $J_1=\frac{1}{\omega_3}$, $J_2(z,\bar{\zeta})=\sum_{k=1}^{\infty} \frac{(\bar{\zeta} z)^k}{\omega_{2 k+3}}$ and $J_3(z,\bar{\zeta})=\sum_{k=1}^{\infty} \frac{k\left(\omega_{2 k+1}-\omega_{2 k+3}\right)}{\omega_{2 k+1} \omega_{2 k+3}}(\bar{\zeta} z)^k.$
Notice that Lemma \ref{Dhat}(iv) implies
\begin{equation}\label{HL10}
k\left(\omega_{2 k+1}-\omega_{2 k+3}\right)=k \int_0^1 s^{2 k+1}\left(1-s^2\right) \omega(s) \mathrm{d} s \lesssim \omega_{2 k+1}, \quad\forall k \in
\mathbb{N}.
\end{equation}
\eqref{1-zzeta} together with Fubini's theorem, Lemma \ref{Dhat}(v) and \cite[Lemma 4]{PR3} yields
$$
\begin{aligned}
\int_{\mathbb{D}}\left|J_2(z, \bar{\zeta})\right|^2 \omega(\zeta) \mathrm{d} A(\zeta)
&\lesssim \int_0^1 \sum_{k=1}^{\infty} \frac{1}{\omega_{2 k+3}^2}r^{2k}|z|^{2 k} \omega(r)\mathrm{d}r\\
& \lesssim \int_0^1\int_0^{|z|r}\frac{\mathrm{d}t}{\widehat{\omega}(t)^2(1-t)^2}\omega(r)\mathrm{d}r+1\\
&=\int_0^{|z|}\frac{1}{\widehat{\omega}(t)^2(1-t)^2}\left(\int_{t/|z|}^1 \omega(r)\mathrm{d}r\right)  \mathrm{d}t+1\\
&\lesssim \frac{1}{\widehat{\omega}(z)(1-|z|)} , \quad z \in \mathbb{D}.
\end{aligned}
$$
Similar argument combining \eqref{HL10} gives
$$\int_{\mathbb{D}}\left|J_3(z, \bar{\zeta})\right|^2 \omega(\zeta) \mathrm{d} A(\zeta)
\lesssim \frac{1}{\widehat{\omega}(z)(1-|z|) }, \quad z \in \mathbb{D}.
$$
Then it follows that
\begin{eqnarray}\label{HL11}
I_1(z)
 &\lesssim &\int_{\mathbb{D}}\left|J_1\right|^2 \omega(\zeta) \mathrm{d} A(\zeta)+\int_{\mathbb{D}}\left|J_2(z, \bar{\zeta})\right|^2
 \omega(\zeta) \mathrm{d} A(\zeta)+ \int_{\mathbb{D}}\left|J_3(z, \bar{\zeta})\right|^2 \omega(\zeta) \mathrm{d} A(\zeta)\nonumber\\
 &\lesssim &\frac{\omega(\mathbb{D})}{\omega_3^2}+\frac{1}{\widehat{\omega}(z)(1-|z|)}
 \lesssim \frac{1}{\widehat{\omega}(z)(1-|z|)}, \quad z \in \mathbb{D}.
\end{eqnarray}
Next we turn to estimate $I_2(z)$. By Lemmas \ref{Dhat}(ii), \ref{Ne1} and the choice of $\varepsilon$, we have
$$
\begin{aligned}
I_2(z)
&\lesssim\int_{0}^{|z|} \frac{\widehat{\omega}(t)^{\frac{\varepsilon}{2-\varepsilon}}}{(1-t)^{\frac{4(1-\varepsilon)}{2-\varepsilon}}}\mathrm{d}t+1\\
&\lesssim\left(\frac{\widehat{\omega}(z)}{(1-|z|)^\beta}\right)^{\frac{\varepsilon}{2-\varepsilon}}\int_{0}^{|z|}\frac{1}{(1-t)^{\frac{4-(\beta+4)\varepsilon}{2-\varepsilon}}}\mathrm{d}t\\
&\lesssim \left(\frac{\widehat{\omega}(z)^{\frac{\varepsilon}{2}}}{(1-|z|)^{1-\frac{3\varepsilon}{2}}}\right)^{\frac{2}{2-\varepsilon}}, \quad z \in
\mathbb{D}.
\end{aligned}
$$
This estimate together with \eqref{HL8} and \eqref{HL11} yields
\begin{equation}\label{HL9}
\left|\int_{\mathbb{D}}f(\zeta)\overline{\zeta (B_z^\omega)^{\prime}(\zeta)}\left(\log\frac{1-\bar{\zeta}z}{1-|z|}\right)\omega(\zeta)\mathrm{d}A(\zeta)\right|
\lesssim \frac{1}{1-|z|},\quad z \in \mathbb{D}.
\end{equation}
Observe that
$$
\log\frac{1}{1-|z|}=\log\frac{1-\bar{\zeta}z}{1-|z|}+\log \frac{1}{1-\bar{\zeta}z},\quad z,\zeta \in \mathbb{D}.
$$
Then \eqref{TL2} combined with \eqref{HL7} and \eqref{HL9} gives
$$
\Big(\log\frac{1}{1-|z|}\Big)(1-|z|)\left|\int_{\mathbb{D}}f(\zeta)(B_{\zeta}^\omega)^{\prime}(z)\omega(\zeta)\mathrm{d}A(\zeta)\right|\lesssim 1,
\quad{|z|\rightarrow 1^-}.
$$
Moreover, by the reproducing formula, we deduce
$$
\sup_{z\in\mathbb{D}}|f^{\prime}(z)|(1-|z|^2)\log\frac{1}{1-|z|}<\infty,
$$
which combined with $f\in \mathcal{B}$ shows that $f\in\mathcal{LB}$, as desired.

\hspace{10pt}

Sufficiency. Assume that $f\in \mathcal{LB}$. Then it is obvious that $f\in\mathcal{B}$. Define
$$
S_f^\omega(h)(z)=\int_{\mathbb{D}}\Big(f(\zeta)-f(z)\Big)\overline{B_z^\omega(\zeta)}h(\zeta)\omega(\zeta)\mathrm{d}A(\zeta),\quad h\in L^\infty,\,\, z\in \mathbb{D}.
$$
Observe that
$$
\begin{aligned}
(1-|z|^2)(S_f^\omega(h))^{\prime}(z)
=&-\int_{\mathbb{D}}f^{\prime}(z)(1-|z|^2)\overline{B_z^\omega(\zeta)}h(\zeta)\omega(\zeta)\mathrm{d}A(\zeta)\\
&+(1-|z|^2)\int_{\mathbb{D}}\Big(f(\zeta)-f(z)\Big){(B_\zeta^\omega)^\prime(z)} h(\zeta)\omega(\zeta)\mathrm{d}A(\zeta),\quad z\in \mathbb{D}.
\end{aligned}
$$
By Lemma \ref{Ne1}, we have
\begin{eqnarray}\label{H1}
\int_{\mathbb{D}}|f^{\prime}(z)|(1-|z|^2)|B_z^\omega(\zeta)||h(\zeta)|\omega(\zeta)\mathrm{d}A(\zeta)
&\leq&|f^{\prime}(z)|(1-|z|^2)\|h\|_{\infty}\|B_z^\omega\|_{A_\omega^1}\nonumber\\
&\asymp&|f^{\prime}(z)|(1-|z|^2)\Big(\log\frac{e}{1-|z|}\Big)\|h\|_{\infty}\nonumber\\
&\lesssim& \|h\|_{\infty},\quad z\in \mathbb{D}.
\end{eqnarray}
By Lemma \ref{Dweight1}, there exists $\gamma_0=\gamma_0(\omega)>0$ such that $\omega_{[-\gamma]}\in \mathcal{D}$ for each $\gamma\in (0,\gamma_0]$. Take $0<\gamma<\min\{\frac{1}{2+\beta},\frac{\gamma_0}{1+\gamma_0}\}$, where $\beta$ is chosen as in
Lemma \ref{Dhat}. By Lemma \ref{Bloch} and \eqref{beta}, we get
$$
\begin{aligned}
&\quad\int_{\mathbb{D}}|f(\zeta)-f(z)||{(B_\zeta^\omega)^\prime(z)}||h(\zeta)|\omega(\zeta)\mathrm{d}A(\zeta)\\
&\lesssim\|h\|_{\infty}\|f\|_{\mathcal{B}}\int_{\mathbb{D}}\beta(\zeta,z)|(B_\zeta^\omega)^\prime(z)|\omega(\zeta)\mathrm{d}A(\zeta)\\
&\lesssim\frac{\|h\|_{\infty}\|f\|_{\mathcal{B}}}{(1-|z|)^\gamma}\int_{\mathbb{D}}|1-\bar{\zeta}z|^{2\gamma}|(B_\zeta^\omega)^\prime(z)|
\omega_{[-\gamma]}(\zeta)\mathrm{d}A(\zeta),\quad{z\in\mathbb{D}}.
\end{aligned}
$$
Combining the above inequality with the estimate in \cite{PR1}, we obtain
\begin{equation}\label{H2}
\int_{\mathbb{D}}|f(\zeta)-f(z)||{(B_\zeta^\omega)^\prime(z)}||h(\zeta)|\omega(\zeta)\mathrm{d}A(\zeta)\lesssim\frac{\|h\|_{\infty}}{1-|z|},\quad{z\in\mathbb{D}}.
\end{equation}
It follows from \eqref{H1} and \eqref{H2} that $S_h^\omega(f)\in\mathcal{B}$ with
\begin{equation}\label{H3}
\|S_h^\omega(f)\|_{\mathcal{B}}\lesssim\|h\|_{\infty}.
\end{equation}

For $\omega\in\mathcal{D}$, it follows from Lemma \ref{Bloch3} and \eqref{H3} that
$$
\left|\int_{\mathbb{D}}g(u)\overline{S_f^\omega(h)(u)}\omega(u)\mathrm{d}A(u)\right|
\leq\|g\|_{A_\omega^1} \|S_f^\omega(h)\|_{\mathcal{B}}
\lesssim \|g\|_{A_\omega^1} \|h\|_{\infty},\quad{g\in H^\infty, h\in L^\infty}.
$$
From the above inequality and Fubini's theorem, we deduce
$$
\begin{aligned}
&\left|\int_{\mathbb{D}}H_{\bar{f}}^\omega(g)(\zeta)\overline{h(\zeta)}\omega(\zeta)\mathrm{d}A(\zeta)\right|\\
=&\left|\int_{\mathbb{D}}\int_{\mathbb{D}}\Big(\overline{f(\zeta)}-\overline{f(u)}\Big)B_u^\omega(\zeta)\overline{h(\zeta)}\omega(\zeta)\mathrm{d}A(\zeta)
g(u)\omega(u)\mathrm{d}A(u)\right|\\
=&\left|\int_{\mathbb{D}}g(u)\overline{S_f^\omega(h)(u)}\omega(u)\mathrm{d}A(u)\right|\\
\lesssim& \|g\|_{A_\omega^1} \|h\|_{\infty},\quad{g\in H^\infty, h\in L^\infty},
\end{aligned}
$$
which gives
$$
\|H_{\bar{f}}^\omega(g)\|_{L_\omega^1}\lesssim\|g\|_{A_\omega^1},\quad{g\in H^\infty}.
$$
A standard density argument shows that $H_{\bar{f}}:A_\omega^1\rightarrow L_\omega^1$ is bounded, completing the proof.
\qed

Notice that for a bounded symbol $f$, if the Toeplitz operator $T_f^\omega$ is bounded on $A_{\omega}^1$, then the Hankel operator $H_f^\omega$ is also bounded on $A_{\omega}^1$. Actually, we have
\begin{equation}\label{H5}
\|H_f^\omega (g)\|_{L_\omega^1}\leq\|fg\|_{L_\omega^1}+\|T_f^\omega(g)\|_{A_\omega^1}\leq(\|f\|_\infty+\|T_f^\omega\|_{A_\omega^1})\|g\|_{A_\omega^1},\quad{g\in A_\omega^1}.
\end{equation}

We end this section with a proposition showing that $H_f^\omega$ is bounded from $A_\omega^1$ to $L_\omega^1$ if the symbol $f\in \mathrm{BMO}_{\omega,\log}^p$ admits a particular form.
 \begin{proposition}\label{HBMO}
Let $\omega\in\mathcal{D}$ and $1<p<\infty$. If $f \in \mathrm{BMO}_{\omega,\log}^p$ such that $f=f_1+f_2$ with $f_1\in \mathrm{BO}_{\log } \cap L^{\infty}$ and
$f_2\in \mathrm{BA}^p_{\omega,\log}$, then $H_f^\omega: A_\omega^1 \rightarrow L_\omega^1$ is bounded.
\end{proposition}
\begin{proof}
It suffices to show that both $H_{f_1}^\omega$ and $H_{f_2}^\omega$ are bounded from $A_\omega^1$ to $L_\omega^1$.
By \eqref{H5}, it follows directly from Proposition \ref{main3} that $H_{f_1}^\omega:A_\omega^1\rightarrow L_\omega^1$ is bounded.

Let $M_f$ be the multiplication operator by $f$. Suppose $1<p<\infty$. We first show that for $f_2\in \mathrm{BA}^p_{\omega,\log}$, the multiplication operator $M_{f_2}: A_\omega^1\rightarrow L_\omega^1$ is bounded.
Obviously $f_2\in \mathrm{BA}_\omega^p$. From Lemma \ref{Bloch2}, we have $\|P_\omega(f_2)\|_{\mathcal{B}}\lesssim \|f_2\|_{\mathrm{BA}_\omega^p}.$ Notice that
for $h\in L^\infty$, $f_2h\in \mathrm{BA}_\omega^p$ and $\|f_2h\|_{\mathrm{BA}_\omega^p}\leq \|h\|_{\infty}\|f_2\|_{\mathrm{BA}_\omega^p}$.
For $g\in A_\omega^1$ and $h\in L^\infty$,  it follows from the reproducing formula for $A_\omega^1$ and Lemma \ref{Bloch3} that
$$
\begin{aligned}
|\langle M_{f_2}(g),h\rangle_{L_\omega^2}|
&=|\langle P_\omega(g),\bar{f_2}h \rangle_{L_\omega^2}|
=|\langle g, P_\omega(\bar{f_2}h) \rangle_{L_\omega^2}|
\leq \|P_\omega(\bar{f_2}h)\|_{\mathcal{B}}\|g\|_{A_\omega^1}\\
&\lesssim \|\bar{f_2}h\|_{\mathrm{BA}_\omega^p} \|g\|_{A_\omega^1}
\leq \|h\|_{\infty}\|f_2\|_{\mathrm{BA}_\omega^p}\|g\|_{A_\omega^1},
\end{aligned}
$$
which shows $\|M_{f_2}\|_{A_\omega^1\rightarrow L_\omega^1}\lesssim \|f_2\|_{\mathrm{BA}_\omega^p}.$

Since $\mathrm{BA}^p_{\omega,\log}\subseteq \mathrm{BA}^1_{\omega,\log}$ for $p>1$, Lemma \ref{BAlog} implies that $T_{f_2}^\omega$ is bounded on $A_\omega^1$. Then $H_{f_2}^\omega=M_{f_2}-T_{f_2}^\omega$ is bounded from $A_\omega^1$ to $L_\omega^1$ too.
This completes the proof.
\end{proof}

For $1<p<\infty$, the problem whether $H_f^\omega:A_\omega^1\rightarrow L_\omega^1$ is bounded for every $f\in \mathrm{BMO}_{\omega,\log}^p$ remains open even in the classical Bergman space.

\end{document}